\documentclass[10pt,twoside,reqno]{amsart}
\usepackage{amsmath}
\usepackage{amssymb}
\usepackage{amsfonts}
\usepackage{amsthm}
\usepackage{graphicx}
\usepackage[initials]{amsrefs}
\usepackage{fancyhdr}
\usepackage{hyperref}

\usepackage{tikz}
\usepackage{tikz-cd}

\usepackage[all,cmtip]{xy}

\newtheorem{theorem}{Theorem}[section]
\newtheorem{corollary}[theorem]{Corollary}
\newtheorem{lemma}[theorem]{Lemma}
\newtheorem{proposition}[theorem]{Proposition}
\theoremstyle{definition}
\newtheorem{definition}[theorem]{Definition}
\newtheorem{example}[theorem]{Example}
\newtheorem{remark}[theorem]{Remark}

\begin{document}

\title[Finite Polyhedral Models]{\large Finite Polyhedral Models for the Space of Equivalent Norms}
\author[Acosta-Portilla Juan Rafael]{Juan Rafael Acosta-Portilla$^{1}$}
\date{May 2026}
\maketitle

\begin{center}
{\footnotesize
$^{1}$Instituto de Investigaciones y Estudios Superiores Económicos y Sociales, Universidad Veracruzana, México\\
E-mail: juaacosta@uv.mx\\
}
\end{center}

\bigskip

{\footnotesize
\noindent
{\bf Abstract.}
We study finite polyhedral models inside the projectivized space of equivalent norms $\mathcal{N}'(X)$ on a finite-dimensional Banach space $X$, endowed with the logarithmic distortion metric. Given a finite symmetric direction set $E$, we introduce the class of $\alpha E$-norms, defined as Minkowski functionals of symmetric polytopes whose vertices lie on the rays prescribed by $E$. We identify the admissible weights $\alpha$ for which this parametrization is non-redundant and prove that they give a one-to-one parametrization of the corresponding finite direction model $\mathcal{N}(E)$.
We then refine this description by introducing complete, symmetric, simplicial $E$-fans, which encode the conical regions on which the associated polyhedral norms are linear. For a fixed fan $\mathcal{F}$, we define the corresponding geometric and coordinate fan models $\mathcal{N}(\mathcal{F})$ and $\mathcal{R}(\mathcal{F})$, and show that they are naturally isomorphic as cones. After quotienting by positive scalar multiplication, the coordinate models become isometric to the corresponding norm models: the logarithmic distortion metric on $\mathcal{N}'(E)$ and $\mathcal{N}'(\mathcal{F})$ is represented exactly by the Hilbert projective metric on the admissible weight spaces.
Finally, we prove completeness results for these finite models and show how they provide finite-dimensional polyhedral approximations to the metric geometry of $\mathcal{N}'(X)$.


}

\bigskip

\section{Introduction}

Banach space geometry and renorming theory have traditionally focused on the study of geometric properties of individual norms. This point of view has led to a rich theory describing phenomena such as smoothness, convexity, rotundity, polyhedrality, and fixed point properties \cite{agarwal2009fixedpointbook, GoebelKirk1990TopicsFixedPoint, godefroy2001renormings, guirao2022renormings, kirk2002handbook}. However, once a Banach space $X$ is fixed, it is also natural to regard the collection of all equivalent norms on $X$ as a geometric object in its own right. This leads to the following general question:

\begin{center}
\emph{What is the structure of the space of equivalent norms?}
\end{center}

Let $X$ be a Banach space and let $\mathcal{N}(X)$ denote the family of equivalent norms on $X$. Since many geometric properties of a norm are invariant under multiplication by a positive scalar, it is natural to identify collinear norms. Thus we consider the projectivized space

\begin{equation}
\mathcal{N}'(X)=\mathcal{N}(X)/\sim,
\end{equation}

\noindent where $\|\cdot\|_1\sim \|\cdot\|_2$ if there exists $c>0$ such that $\|\cdot\|_2=c\|\cdot\|_1$. On this quotient, one may consider the logarithmic distortion metric

\begin{equation}
d\left(\overline{\|\cdot\|}_1,\overline{\|\cdot\|}_2\right)=\log\frac{u}{l},
\end{equation}

\noindent where $u\geq l>0$ are the optimal constants satisfying

\begin{equation}
l\|x\|_1\leq \|x\|_2\leq u\|x\|_1
\end{equation}

\noindent for all $x\in X$. This metric is a logarithmic version of the usual Banach--Mazur distortion between norms and gives a natural way to compare the shapes of their unit balls. Several related metrics and topological structures on spaces of equivalent norms have appeared in the literature \cite{TomczakJaegermann1989BanachMazurDA, dominguezPhoti2008porosity, dominguezPhoti2010genericityinsomebanach, fabianZajicekZizler1982residualityRotundNorms}. Nevertheless, the global metric geometry of $\mathcal{N}'(X)$ remains difficult to describe.

The main difficulty is that $\mathcal{N}'(X)$ is too large to admit a simple global finite-dimensional parametrization. Even in finite-dimensional spaces, the unit ball of an equivalent norm may be an arbitrary symmetric convex body with nonempty interior. Thus, instead of looking for a single global coordinate system, the present paper develops finite polyhedral models inside $\mathcal{N}'(X)$. These models are determined by finite sets of prescribed directions and by the polyhedral regions on which the corresponding norms are linear.

More precisely, let $X$ be finite-dimensional and let $E\subset X$ be a finite symmetric direction set. The set $E$ prescribes the possible directions in which the extreme points of a polyhedral unit ball may occur. We consider the family $\mathcal{N}(E)$ of all polyhedral norms whose extreme points lie on rays determined by $E$. This gives a finite direction model in the space of norms. However, the choice of $E$ alone does not determine a norm. One also needs radial data describing how far the unit ball extends in each prescribed direction. This leads to the notion of an $E$-admissible multiindex $\alpha$, and to the associated $\alpha E$-norm defined as the Minkowski functional of the symmetric polytope

\begin{equation}
\operatorname{conv}(\alpha E)=\operatorname{conv}\{\pm \alpha_1^{-1}x_1,\ldots,\pm \alpha_n^{-1}x_n\}.
\end{equation}

\noindent The admissibility condition removes the redundancy in this parametrization: it ensures that each prescribed direction is radially calibrated with the boundary of the associated polytope. In this way, the coordinate model $\mathcal{R}(E)$ of $E$-admissible multiindices gives a non-redundant parametrization of the finite direction model $\mathcal{N}(E)$.

A second layer of structure appears when one studies the regions of linearity of these norms. A polyhedral norm is linear on the cones generated by the faces of its unit ball. Therefore, in order to describe its local linear structure, one is naturally led to complete, symmetric, simplicial fans. Given a complete, symmetric, simplicial $E$-fan $\mathcal{F}$, we consider the subfamily $\mathcal{N}(\mathcal{F})$ of norms which are linear on every cone of $\mathcal{F}$. The corresponding coordinate model $\mathcal{R}(\mathcal{F})$ consists of those weights on $E$ which are compatible with the fixed fan. Thus the finite direction model records the allowed vertex directions, while the finite fan model records the conical regions of linearity.

The main results of the paper can be summarized as follows. First, we prove that every norm in $\mathcal{N}(E)$ is uniquely represented by an $E$-admissible multiindex. Equivalently, the map

\begin{equation}
\alpha \longmapsto \rho(\cdot;\alpha E)
\end{equation}

\noindent establishes a bijection between the coordinate direction model $\mathcal{R}(E)$ and the geometric direction model $\mathcal{N}(E)$. Second, we show that compatible fans provide a piecewise linear representation of $\alpha E$-norms: on each maximal cone of a compatible fan, the norm is given by a linear expression determined by the corresponding weights. Third, we prove that for a fixed fan $\mathcal{F}$, the coordinate fan model $\mathcal{R}(\mathcal{F})$ and the geometric fan model $\mathcal{N}(\mathcal{F})$ are isomorphic as cones.

The metric structure is also finite-dimensional. After passing to the quotient by positive scalar multiplication, the parametrization by admissible multiindices becomes an isometry. More precisely, if $\mathcal{R}'(E)$ denotes the projectivization of $\mathcal{R}(E)$ endowed with the Hilbert metric $d_H$, then the map

\begin{equation}
\phi'_E:(\mathcal{R}'(E),d_H)\longrightarrow(\mathcal{N}'(E),d)
\end{equation}

\noindent is an isometry. The same holds for the fan models associated with a fixed complete, symmetric, simplicial $E$-fan $\mathcal{F}$. Thus the logarithmic distortion metric on finite polyhedral models of norms is exactly represented by the Hilbert projective metric on the corresponding coordinate models.

Finally, we study the completeness and approximation properties of these finite models. The admissible coordinate models are closed under the relevant metric limits, and their images provide complete finite-dimensional submodels of $\mathcal{N}'(X)$. Moreover, in finite-dimensional spaces, polyhedral norms are dense in the space of equivalent norms with respect to the logarithmic distortion metric. Hence the global space $\mathcal{N}'(X)$ can be studied through an increasing family of finite polyhedral models. In this sense, the present work provides a finite-dimensional approximation scheme for the geometry of the space of equivalent norms.

The paper is organized as follows. Section 2 recalls the basic notions concerning equivalent norms, Minkowski functionals, polyhedral norms, fans, and the Hilbert projective metric. In Section 3 we introduce direction sets, $\alpha E$-norms, and $E$-admissible multiindices, and we prove the basic parametrization theorem for polyhedral norms with prescribed directions. Section 4 develops the local fan structure of $\alpha E$-norms and introduces admissible pairs $(\mathcal{F},\alpha)$. In Section 5 we define the finite direction, coordinate, and fan models, and prove their cone and isometric structures. Section 6 is devoted to completeness and density properties, showing how the finite polyhedral models fit into the global metric space of equivalent norms.

\section{Preliminaries}

In this section we introduce the basic notions and fix the notation used throughout the paper.

\subsection*{Families of equivalent norms}

Given a Banach space $X$, we denote by $\mathcal{N}(X)$ the family of equivalent norms. We say that $\| \cdot \|_1 \sim \| \cdot \|_2$ if there exists $c > 0$ such that $\| \cdot \|_2 = c \| \cdot \|_1$. We denote 

\begin{equation}\label{quotient of collinear norms}
\mathcal{N}'(X) = \mathcal{N}(X)/\sim,
\end{equation}

\noindent the family of equivalent norms modulo collinearity. For each $\| \cdot \|_1 , \| \cdot \|_2 \in \mathcal{N}(X)$ there exist sharp $u \geq l > 0$ such that

\begin{equation}
l \| x \|_1 \leq \| x \|_2 \leq u \| x\|_1
\end{equation}
\noindent for each $x \in X$.

The comparison between norms is naturally encoded by a logarithmic metric measuring their distortion. On the quotient space $\mathcal{N}'(X)$ we define

\begin{equation}\label{definition distance d}
d \left(\overline{\| \cdot \|}_1 , \overline{\| \cdot \|}_2\right) = \log \frac{u}{l}
\end{equation}

\noindent A detailed discussion of this metric can be found in \cite{dominguezPhoti2008porosity, dominguezPhoti2010genericityinsomebanach}.

\subsection*{Minkowski functional}

For a fixed $K \subset X$, the Minkowski functional is the extended real valued function

\begin{equation}
\rho (x ; K) = \inf \{r> 0 \mid x \in r K \}.
\end{equation}

\noindent In particular, $\rho( \, \cdot \, ; \operatorname{conv}(K))$ is a norm when $K$ is a finite symmetric set on a finite-dimensional space and $\operatorname{conv}(K)$ has nonempty interior.

\subsection*{Basics on simplex}

The following definitions and results can be found in the classical references by Schrijver \cite{Schrijver1998}, Ziegler \cite{ziegler1995polytopes}, and Klee \cite{Klee1959SomeCO}. By a polytope in $\mathbb{R}^n$ we mean the convex hull of finitely many points. A face of a polytope $P$ is the intersection of $P$ with a supporting hyperplane, and a facet is a maximal proper face with respect to inclusion. Every face of a polytope is itself a polytope, and a vertex is an extreme point. A $d$-simplex $\Delta^d$ is the convex hull of points ${x_0, \dots, x_d}$ such that ${x_1 - x_0, \dots, x_d - x_0}$ is linearly independent. For any $x \in \Delta^d$, its barycentric coordinates are the unique scalars $c_i \geq 0$ with $\sum c_i = 1$ such that

\begin{equation}
x = \sum_{i = 0}^d c_i x_i.  
\end{equation}   

\noindent The dimension of a polytope is the dimension of its affine hull. A triangulation of a $d$-dimensional polytope $P$ with vertex set ${x_1, \dots, x_n}$ is a collection of $d$-simplices $\Delta_i^d$ with vertices in ${x_1, \dots, x_n}$ such that $P = \bigcup_i \Delta_i^d$ and, for $i \neq j$, the intersection $\Delta_i^d \cap \Delta_j^d$ is a simplex of dimension less than $d$. It is well known that every polytope admits a triangulation \cite{deloera2010triangulation, ziegler1995polytopes}.

\subsection*{Basics on Fans}

The notions of polyhedral cones and fans used below are standard in polyhedral and toric geometry; see, for instance, \cite{CLS2011ToricVarieties,ziegler1995polytopes}.
A polyhedral cone in a finite-dimensional vector space $X$ is a set of the form

\begin{equation}
C = \operatorname{cone}(x_1,\dots,x_k) = \left\{ \sum_{i=1}^k \lambda_i x_i \, \middle| \, \lambda_i \ge 0 \right\}.
\end{equation}

\noindent A fan $\mathcal{F}$ in $X$ is a finite collection of polyhedral cones such that:

\begin{itemize}
\item[1)] If $C \in \mathcal{F}$ and $G$ is a face of $C$, then $G \in \mathcal{F}$.
\item[2)] For any $C_1, C_2 \in \mathcal{F}$, the intersection $C_1 \cap C_2$ is a face of both $C_1$ and $C_2$.
\end{itemize}

\noindent A fan $\mathcal{F}$ is called simplicial if every cone $C \in \mathcal{F}$ is generated by a linearly independent set of vectors. The fan $\mathcal{F}$ is said to be complete if

\begin{equation}
X= \bigcup_{C \in \mathcal{F}} C.
\end{equation}

\begin{lemma}\label{lema full dimension of generators of maximal cones on fan}
Let $X$ be a $d$-dimensional Banach space and $\mathcal{F}$ be a complete, simplicial fan in $X$. Then every maximal cone $C\in\mathcal{F}$ is generated by exactly $d$ linearly independent rays.
\end{lemma}

\subsection*{Polyhedral norms}

The following characterization can be found in Klee \cite{Klee1960PoSec}. A norm $\|\cdot\|$ on a Banach space $X$ is said to be polyhedral if, for every finite-dimensional subspace $Y\subset X$, there exist functionals $x_1^*,\dots,x_n^*\in S_{X^*}$ such that, for every $y\in Y$,

\begin{equation}\label{polyhedral norm definition}
\|y\|=\max_{1\leq i\leq n}|x_i^*(y)|.
\end{equation}

\noindent This is a dual, or functional, characterization of polyhedrality. Since the present paper is concerned with finite-dimensional spaces, we shall use the equivalent geometric formulation: a norm is polyhedral if and only if its unit ball is a symmetric polytope with nonempty interior. Equivalently, the norm is the Minkowski functional of such a polytope.

\subsection*{Hilbert Metric}
\noindent We recall the classical Hilbert metric, see \cite{LemmensNussbaum2012, Bushell1973hilbert'sMetric}. For $x,y\in \mathbb{R}^k_{>0}$, we define

\begin{equation}
M(x,y):=\max_{1\leq i\leq k}\frac{x_i}{y_i}
\end{equation}

\noindent and

\begin{equation}
m(x,y):=\min_{1\leq i\leq k}\frac{x_i}{y_i}.
\end{equation}

\noindent Since the quotient $\frac{M(x,y)}{m(x,y)}$ is invariant under positive rescaling of $x$ and $y$, it defines a metric on the projective space

\begin{equation}\label{Proyective positive cone}
\mathbb{P}(\mathbb{R}^k_{>0})=\mathbb{R}^k_{>0}/\sim, \qquad x\sim y \iff y=\lambda x \text{ for some } \lambda>0.
\end{equation}

\noindent The Hilbert metric is then defined on $\mathbb{P}(\mathbb{R}^k_{>0})$ by

\begin{equation}
d_H([x],[y])=\log\left(\frac{M(x,y)}{m(x,y)}\right).
\end{equation}

The following classical results can be found in \cite{LemmensNussbaum2012}.

\begin{proposition}\label{proposition completeness of projective positive cone}
The metric space $(\mathbb{P}(\mathbb{R}^k_{>0}),d_H)$ is complete. 
\end{proposition}

\begin{proposition}\label{proposition normalized representatives converge}
Let $(\overline{\alpha}_n)$ be a sequence in $(\mathbb{P}(\mathbb{R}^k_{>0}),d_H)$ converging to $\overline{\alpha}$. Choose representatives $\alpha_n=(\alpha_1^n,\dots,\alpha_k^n)\in\overline{\alpha}_n$ and $\alpha=(\alpha_1,\dots,\alpha_k)\in\overline{\alpha}$ such that

\begin{equation}
\sum_{i=1}^k \alpha_i^n=1, \qquad\sum_{i=1}^k \alpha_i=1.
\end{equation}

\noindent Then

\begin{equation}
\lim_{n\to\infty}\alpha_i^n=\alpha_i, \qquad i=1,\dots,k.
\end{equation}
\end{proposition}

\section{Direction sets and $\alpha E$-norms}

The purpose of this section is to introduce a finite-dimensional parametrization of polyhedral norms through sets of directions and associated weights.

\begin{definition}[Direction set]
Let $(X,\|\cdot\|)$ be a finite-dimensional Banach space. A set $E\subset S_X$ is called a finite symmetric direction set, or simply a direction set, if $E$ is finite, $E=-E$, and $\operatorname{conv}(E)$ has nonempty interior. After choosing one representative from each antipodal pair, we write

\begin{equation}
E=\{\pm x_1,\dots,\pm x_n\}.
\end{equation}
\end{definition}

\noindent The elements of $E$ should be understood as the prescribed directions in which the vertices of the associated polyhedral unit balls may occur. We shall use the same symbol $E$ both for the symmetric set and, after fixing representatives of antipodal pairs, for the corresponding ordered list

\begin{equation}
E = (x_1 , \dots, x_n, -x_1, \dots, -x_n).
\end{equation}

\noindent We denote by $\operatorname{Dir}(E)$ the set of positive rays generated by $E$,

\begin{equation}
\operatorname{Dir}(E)=\{cx \mid c>0, x \in E\}.
\end{equation}

\noindent For a convex set $B\subset X$, we denote by $\operatorname{Ext}(B)$ the set of extreme points of $B$.

\noindent For the purpose of simplifying the calculations, we will introduce the following notation: If $\alpha= (\alpha_1 , \dots , \alpha_n)$ is a multiindex and $E = (x_1 , \dots , x_n , - x_1 , \dots , -x_n)$ is a direction set, then we define

\begin{equation}
\alpha E := (\alpha^{-1}_1 x_1 , \dots , \alpha^{-1}_n x_n , -\alpha^{-1}_1 x_1 , \dots , -\alpha^{-1}_n x_n)
\end{equation}

\noindent and

\begin{equation}
\operatorname{conv}(\alpha E) = \operatorname{conv} \{\pm \alpha^{-1}_1 x_1 , \dots , \pm \alpha^{-1}_n x_n \}
\end{equation}

\noindent In addition we will use the partial order $(\alpha_1 , \dots , \alpha_n) \leq (\beta_1 , \dots , \beta_n)$ if $\alpha_i \leq \beta_i$ for each $i=1 , \dots , n$.

The following construction allows us to encode norms using finite geometric data. 

\begin{definition}[$\alpha E$-norm]
If $\alpha > 0$, then the $\alpha E$-norm on $X$ is defined by 

\begin{equation}
\rho(x  ; \alpha  E) = \inf \{ r > 0 \mid x \in r \operatorname{conv}(\alpha E)\}.
\end{equation}
\end{definition}

\noindent The intended meaning is that $\rho(\, \cdot \, ;\alpha E)$ should make $x_i$ have length $\alpha_i$. This may fail if the point $\alpha_i^{-1}x_i$ is not exposed on the radial boundary of $\operatorname{conv}(\alpha E)$; moreover, different weights may define the same norm. The admissibility condition below removes this redundancy.

\begin{definition}[$E$-admissible multiindex]
We say that a multiindex $\alpha=(\alpha_1, \dots , \alpha_n)> 0$ is $E$-admissible if for every $ 0<\varepsilon_i < \alpha_i$, $i = 1 , \dots , n$ and $\alpha' = (\alpha_1 , \dots , \alpha_i - \varepsilon_i , \dots , \alpha_n)$ we have that 

\begin{equation}
\rho( \, \cdot \, ;\alpha'  E ) \neq \rho(\, \cdot \, ; \alpha  E)
\end{equation}
\end{definition}

\noindent Intuitively, $\alpha$ is $E$-admissible when the finite family $\alpha E$ is radially calibrated with its own convex hull. In every prescribed direction $x_i$, the point $\alpha_i^{-1}x_i$ lies exactly on the radial boundary of $\operatorname{conv}(\alpha E)$, 
or equivalently

\begin{equation}
\rho(x_i;\alpha E)=\alpha_i.
\end{equation}

\noindent Thus no coordinate of $\alpha$ can be decreased without changing the resulting polyhedral unit ball.

\begin{example}
Let $X = \mathbb{R}^2$, $e_3 = \frac{1}{2}(e_1 + e_2)$ and $E = \{ \pm e_1 , \pm e_2, \pm e_3 \}$. For each $\delta > 0$ we define $\alpha_\delta = (1 , 1, \delta)$. Then for $\delta \geq 1$

\begin{equation}
\rho((x, y);\alpha_\delta E)=                                                                                                                                                                                                                                                                                                                                                                                                                                                                                                                                                                                                                                                                                                                                                                                                                                                                                                                                                                                                                                                                                                                                                                                                                                                                                                                                                                                                                                                                                                                                                                                                                                                                                                                                                         |x|+|y|.
\end{equation}

\noindent  and for $0 < \delta < 1$

\begin{equation}
\rho((x,y);\alpha_\delta E)
=
\begin{cases}
|x|+(2\delta-1)|y|, & xy\geq 0 \text{ and } |x|\geq |y|, \\[2mm]
(2\delta-1)|x|+|y|, & xy\geq 0 \text{ and } |x|\leq |y|, \\[2mm]
|x|+|y|, & xy\leq 0.
\end{cases}
\end{equation}
\end{example}

The next result establishes the basic properties of the $\alpha E$-norms, showing that they define equivalent norms and can be reduced to admissible parameters.

\begin{lemma}\label{lemma properties of E-admissible multiindex1}
Let $X$ be a finite-dimensional Banach space, $E =( x_1 , \dots , x_n , - x_1 , \dots , -x_n )$ be a direction set and $\alpha=(\alpha_1 , \dots , \alpha_n)>0$. Then:

\begin{itemize}
\item[A)] The $\alpha E$-norm is an equivalent norm, $\rho(\, \cdot \, ; \alpha E) \in \mathcal{N}(X)$.
\item[B)] For each $i = 1 , \dots, n $ we have that $\rho(x_i ; \alpha E) \leq \alpha_i$.
\item[C)] Up to replacing $\alpha$ by another multiindex defining the same norm, $\alpha$ may be assumed $E$-admissible. More precisely, there exists an $E$-admissible multiindex $\alpha'$ such that $\rho(\, \cdot \,;\alpha' E)=\rho(\, \cdot \,;\alpha E)$.
\end{itemize}
\end{lemma}
\begin{proof}
$A)$. Let $s = \max \{ \alpha_1 , \cdots , \alpha_n\}$. We note that 

\begin{equation}
\begin{split}
s^{-1} \cdot \operatorname{conv}( E ) & = \operatorname{conv}\{ \pm s^{-1} x_1 , \cdots , \pm s^{-1} x_n \} \\
                  &  \subset \operatorname{conv}\{ \pm \alpha^{-1}_1 x_1 , \cdots , \pm \alpha^{-1}_n x_n \} \\
                  & = \operatorname{conv}(\alpha E).
\end{split}
\end{equation}

\noindent Since $\operatorname{conv}(E)$ has nonempty interior and the fact that $\operatorname{conv}\{ \pm \alpha^{-1}_1 x_1 , \cdots , \pm \alpha^{-1}_n x_n \}$ is a symmetric convex set on a finite-dimensional space, we conclude that $\rho(\, \cdot \, ; \alpha E)$ is an equivalent norm.

$B)$. Since $\alpha^{-1}_i x_i \in \operatorname{conv}(\alpha E)$ and $\rho(x_i ; \alpha E)^{-1} = \sup \{r > 0 \mid r x_i \in \operatorname{conv}(\alpha E) \}$, then $ \rho(x_i ; \alpha E)^{-1} \geq \alpha^{-1}_i$.

$C)$. Let $\varepsilon_i = \sup \{  \varepsilon \in [0 , \alpha_i) \mid  (\alpha_i - \varepsilon)^{-1} x_i \in \operatorname{conv}(\alpha E)  \}$ for each $i = 1 , \dots ,n $. By the compactness of $\operatorname{conv}(\alpha E)$ we have $\varepsilon_i < \alpha_i$ and $(\alpha_i - \varepsilon_i)^{-1} x_i \in \operatorname{conv}(\alpha E)$. Note that $[0 , \alpha^{-1}_ix_i] \subset [0 , (\alpha_i - \varepsilon_i)^{-1} x_i]$. Then $\alpha' = (\alpha_1 - \varepsilon_1, \cdots ,  \alpha_n - \varepsilon_n)$ satisfies 

\begin{equation}
\operatorname{conv}\{ \pm (\alpha_1 - \varepsilon_1)^{-1} x_1 , \cdots , \pm (\alpha_n - \varepsilon_n)^{-1} x_n \} = \operatorname{conv}\{\pm \alpha^{-1}_1 x_1 , \cdots , \pm \alpha^{-1}_n x_n\}.
\end{equation}

\noindent Thus $\rho( \, \cdot \, ;\alpha'  E ) = \rho( \, \cdot \, ;\alpha  E )$ and for any $0< \eta_i < \alpha'_i$ we have that 
$(\alpha'_i - \eta_i)^{-1}x_i \notin \operatorname{conv}(\alpha E)$. That is, 
$\alpha^* =(\alpha'_1 , \dots , \alpha'_i - \eta_i, \dots, \alpha'_n)$ satisfies 

\begin{equation}
\operatorname{conv}(\alpha' E) \subsetneq \operatorname{conv}(\alpha^* E).
\end{equation}

\noindent Then $\alpha'$ is $E$-admissible.
\end{proof}

The following characterization identifies the geometric meaning of admissibility in terms of the position of the vertices of the associated convex body.

\begin{lemma}[Characterization of $E$-admissibility]\label{lemma properties of E-admissible multiindex}
Let $X$ be a finite-dimensional Banach space, $E = ( x_1 , \dots , x_n , - x_1 , \dots , -x_n )$ be a direction set and $\alpha=(\alpha_1 , \dots , \alpha_n) > 0$. Then the following statements are equivalent:

\begin{itemize}
\item[1)] $\alpha$ is $E$-admissible.
\item[2)] For each $0 < \alpha' \leq \alpha$ with $\alpha' \neq \alpha$, we have that $\operatorname{conv}(\alpha' E) \supsetneq \operatorname{conv}(\alpha E)$.
\item[3)] For each $0 < \alpha' \leq \alpha$ with $\alpha' \neq \alpha$, we have that $\rho(\, \cdot \, ; \alpha' E) \neq \rho(\, \cdot \, ; \alpha E)$.
\item[4)] For every $i = 1 , \dots , n$ we have that $\rho(x_i ; \alpha E) = \alpha_i$.
\item[5)] For any $\varepsilon>0$ and $i =1 , \dots,n $ we have that $(\alpha^{-1}_i + \varepsilon) x_i \notin \operatorname{conv}(\alpha E)$.
\item[6)] Every $\alpha^{-1}_i x_i$ is on the boundary of $\operatorname{conv}(\alpha E)$ for $i = 1 , \dots , n $.
\item[7)] There exists a norm $\| \cdot \|$ for $X$ such that $\alpha^{-1}_i x_i \in S_{\| \cdot \|}$ for each $i = 1 , \dots , n $. 
\item[8)] There exists a norm $\| \cdot \|$ for $X$ such that $\| x_i \| = \alpha_i$ for each $i = 1 , \dots , n$
\end{itemize}
\end{lemma}
\begin{proof}
Set $\alpha'= (\alpha_1 - \varepsilon_1 , \cdots, \alpha_n - \varepsilon_n)$ for some $0 \leq \varepsilon_i < \alpha_i $ with $\sum \varepsilon^2_i >0$.

$1)$ implies $2)$. Without loss of generality we may assume that $\varepsilon_1 > 0$ and we define $\alpha^* = (\alpha_1 - \varepsilon_1 , \alpha_2 , \cdots , \alpha_n)$. Since $\alpha^{-1}_1 < (\alpha_1 - \varepsilon_1)^{-1}$, then $\operatorname{conv}(\alpha^* E) \supset \operatorname{conv}(\alpha E)$. On the other hand $\alpha$ is $E$-admissible, then  satisfies $\rho(\, \cdot  \, ; \alpha^* E) \neq \rho(\, \cdot  \, ; \alpha E)$. Thus $\operatorname{conv}(\alpha^* E) \supsetneq \operatorname{conv}(\alpha E)$. Hence $\operatorname{conv}(\alpha' E) \supset \operatorname{conv}(\alpha^* E) \supsetneq \operatorname{conv}(\alpha E)$.

$2)$ implies $3)$. By definition $\rho(\, \cdot \, ; \alpha' E)$ and $\rho(\, \cdot \, ; \alpha E)$ are the respective Minkowski functionals over $\operatorname{conv}(\alpha' E)$ and $\operatorname{conv}(\alpha E)$, thus $\operatorname{conv}(\alpha' E) \supsetneq \operatorname{conv}(\alpha E)$ implies $\rho(\, \cdot \, ; \alpha' E) \neq\rho(\, \cdot \, ; \alpha E)$.

$3)$ implies $1)$ is trivial. Now we prove $3)$ implies $4)$ by a contrapositive argument. By definition, for any $i =1 , \dots , n $ we have that $\alpha^{-1}_i x_i \in \operatorname{conv}(\alpha E)$. Then $\rho(x_i ; \alpha E) \leq \alpha_i$. 
We suppose $\neg 4)$. That is, for some $j =1 , \dots , n $, it is fulfilled $0 < \alpha'_j = \rho(x_j ; \alpha E)  < \alpha_j$. 
Then by the Minkowski functional definition and the compactness of $\operatorname{conv}(\alpha E)$, we have, $(\alpha'_j)^{-1} x_j \in \operatorname{conv}(\alpha E)$. Set $\alpha' = (\alpha_1 , \dots, \alpha_{j-1} , \alpha'_j , \alpha_{j+1} , \dots , \alpha_n)$. Then $0 < \alpha' \leq \alpha$ with $\alpha' \neq \alpha$. 
Since $[0, \alpha^{-1}_jx_j] \subset
[0 , (\alpha'_j)^{-1} x_j]$, we have $\operatorname{conv}(\alpha E) \subset \operatorname{conv}(\alpha' E)$. On the other hand 
$(\alpha'_j)^{-1}x_j \in \operatorname{conv}(\alpha E) $ implies $\operatorname{conv}(\alpha' E) \subset \operatorname{conv}(\alpha E)$. 
Then $\rho(\, \cdot \, ; \alpha' E) = \rho(\, \cdot \, ; \alpha E)$.

$4)$ implies $5)$. Since for each $i =1 , \dots , n $ we have $\rho(x_i ; \alpha E) = \alpha_i$, then for each $0 < \varepsilon_i < \alpha_i$ we have that $x_i \notin (\alpha_i - \varepsilon_i) \operatorname{conv}(\alpha E)$, that is, $(\alpha_i - \varepsilon_i)^{-1} x_i \notin \operatorname{conv}(\alpha E)$. We define $\varepsilon'_i = (\alpha_i - \varepsilon_i)^{-1} - \alpha^{-1}_i$. Thus $(\alpha^{-1}_i + \varepsilon'_i)x_i \notin \operatorname{conv}(\alpha E)$ for each $i = 1 , \dots , n $, note that as $\varepsilon_i$ ranges over $(0,\alpha_i)$, the number $\varepsilon_i'$ ranges over $(0,\infty)$.

$5)$ implies $6)$ is trivial. $6)$ implies $7)$. Let $\| \cdot \| = \rho(\, \cdot \, ; \alpha E)$, by hypothesis $\alpha^{-1}_i x_i \in S_\rho = S_{\| \cdot \|}$.

$7)$ equivalent to $8)$ is trivial. Finally we prove $7)$ implies $1)$. Since $\|\alpha^{-1}_i x_i\| =1$, then $\operatorname{conv}(\alpha E) \subset B_{\| \cdot \|}$ and for each $\varepsilon > 0$ we have that $(\alpha^{-1}_i + \varepsilon) x_i \notin B_{\| \cdot \|} $. Hence $(\alpha^{-1}_i + \varepsilon) x_i \notin \operatorname{conv}(\alpha E)$ and $\alpha^{-1}_i x_i$ is on the boundary of $\operatorname{conv}(\alpha E)$ for each $i = 1 , \dots , n $, that is, for each $0 < \varepsilon_i < \alpha_i$ and $\alpha' = (\alpha_1 , \dots , \alpha_i - \varepsilon_i , \dots , \alpha_n)$ it is fulfilled $\operatorname{conv}(\alpha' E) = \operatorname{conv}( \{\pm \alpha^{-1}_1 x_1 , \dots , \pm (\alpha_i - \varepsilon_i)^{-1}  x_i , \dots \pm \alpha^{-1}_n x_n \} ) \supsetneq \operatorname{conv}(\alpha E)$. Then $\rho(\, \cdot \, ; \alpha' E) \neq \rho(\, \cdot \, ; \alpha E)$.
\end{proof}

\begin{remark}\label{remark non redundancy in representation of alpha E norms}
For a fixed direction set $E = (x_1,\dots,x_n,-x_1,\dots,-x_n)$, the admissibility condition
removes the redundancy in the parametrization. Indeed, if $\alpha$ and $\beta$
are $E$-admissible and

\begin{equation}
\rho(\,\cdot\,;\alpha E)=\rho(\,\cdot\,;\beta E),
\end{equation}

\noindent then

\begin{equation}
\alpha_i=\rho(x_i;\alpha E)=\rho(x_i;\beta E)=\beta_i
\end{equation}

\noindent for every $i=1,\dots,n$. Hence $\alpha=\beta$.
\end{remark}

\begin{theorem}\label{theorem fixed E representation}
Let $X$ be a finite-dimensional Banach space, $E$ be a direction set and $p\in\mathcal{N}(X)$. Then the following statements are equivalent:

\begin{itemize}
\item[1)] $\operatorname{Ext}(B_p)\subset \operatorname{Dir}(E)$.
\item[2)] There exists an $E$-admissible multiindex $\alpha$ such that

\begin{equation}
p(x)=\rho(x;\alpha E), \qquad x\in X.
\end{equation}
\end{itemize}
\end{theorem}
\begin{proof}
We assume $E=(x_1 , \dots, x_n , -x_1 , \dots , - x_n)$. For every $i = 1 , \dots, n$ we define $\alpha_i = p(x_i)$. Then $\|\pm \alpha^{-1}_i x_i\| = 1$. That is, every $\pm \alpha^{-1}_i x_i$ lies on the boundary of $B_p$. Since $\operatorname{Dir}(E) \cap S_p = \{\pm \alpha^{-1}_1 x_1 , \dots , \pm \alpha^{-1}_n x_n \}$, 

\begin{equation}
\operatorname{Ext}(B_p) \subset \{\pm \alpha^{-1}_1 x_1 , \dots , \pm \alpha^{-1}_n x_n \}.
\end{equation}

\noindent By hypothesis

\begin{equation}
B_p = \operatorname{conv}(\operatorname{Ext}(B_p)) = \operatorname{conv}\{\pm \alpha^{-1}_1 x_1 , \dots , \pm \alpha^{-1}_n x_n \} = \operatorname{conv}(\alpha E).
\end{equation}

\noindent Hence the Minkowski functional $\rho(\, \cdot \, ; \alpha E)$ satisfies

\begin{equation}
p(x) = \rho(x; \alpha E) , \qquad x \in X
\end{equation}

\noindent and by Lemma \ref{lemma properties of E-admissible multiindex} $\alpha$ is $E$-admissible.

Conversely. $\rho( \, \cdot \, ; \alpha E)$ is the Minkowski functional of $\operatorname{conv}\{ \pm \alpha^{-1}_1 x_1 , \dots , \pm \alpha^{-1}_n x_n\}$. Then $\operatorname{Ext}(B_p)\subset \operatorname{Dir}(E)$.
\end{proof}

\begin{remark}
Thus, for a fixed direction set $E$, the $E$-admissible multiindices provide a non-redundant parametrization of the polyhedral norms whose 
extreme points lie in the directions prescribed by \(E\).
\end{remark}

\begin{corollary}[Combinatorial–metric representation of polyhedral norms]\label{corollary representation of polyhedral norms}
Let $(X,\|\cdot\|)$ be a finite-dimensional Banach space.

\begin{itemize}
\item[1)] Let $E = (x_1,\dots,x_n,-x_1,\dots,-x_n)$ be a direction set and let $\alpha = (\alpha_1,\dots,\alpha_n) > 0$ be an $E$-admissible multiindex. Then the Minkowski functional
    
\begin{equation}
\rho(\,\cdot\,;\alpha E)
\end{equation}
    
\noindent defines a polyhedral norm on $X$ satisfying

\begin{equation}
\rho(x_i;\alpha E) = \alpha_i, \qquad \text{for each } i=1,\dots,n.
\end{equation}

\item[2)] Conversely, for every polyhedral norm $p$ on $X$, there exist a direction set $E$ and an $E$-admissible multiindex $\alpha$ such that
    
\begin{equation}
p(x) = \rho(x;\alpha E), \qquad \text{for all } x \in X.
\end{equation}
\end{itemize}
\end{corollary}
\begin{proof}
The first assertion follows from Lemma \ref{lemma properties of E-admissible multiindex1} and Lemma \ref{lemma properties of E-admissible multiindex}.

\noindent For the converse, let $p$ be a polyhedral norm on $X$. Since $p$ is polyhedral, its unit ball $B_p=\{x\in X : p(x)\leq 1\}$ is a symmetric polytope with nonempty interior. Hence there exist vectors $v_1,\dots,v_n \in X$ such that

\begin{equation}
B_p=\operatorname{conv}\{\pm v_1,\dots,\pm v_n\}.
\end{equation}

\noindent For each $i=1,\dots,n$, set 

\begin{equation}
x_i = \| v_i \|^{-1} v_i.
\end{equation}

\noindent We define $E=(x_1,\dots,x_n,-x_1,\dots,-x_n)$. Note that $E \subset S_{\| \cdot \|}$, $E$ is symmetric and $\operatorname{conv}(E)$ has nonempty interior. Then $E$ is a direction set and 

\begin{equation}
\operatorname{Ext}(B_p) \subset \{\pm v_1,\dots,\pm v_n\} \subset \operatorname{Dir}(E)
\end{equation}

\noindent Thus by Theorem \ref{theorem fixed E representation} there exists an $E$-admissible multiindex $\alpha$ such that 

\begin{equation}
p(x) = \rho(x ; \alpha E), \qquad x \in X.
\end{equation}
\end{proof}

\section{Local fan structure of $\alpha E$-norms}\label{section:local Fan structure}

In this section we describe the local structure of $\alpha E$-norms through fan decompositions, which provide piecewise linear representations of the norm.

The direction set $E$ determines the possible rays on which the vertices of the unit ball may lie, and an admissible multiindex $\alpha$ determines the corresponding radial positions. However, this information does not explicitly record the regions on which the associated polyhedral norm is linear. These regions are conical and are determined by the facial structure of the polytope $\operatorname{conv}(\alpha E)$. Fans provide a convenient language for encoding this local piecewise linear structure.

\begin{definition}[$E$-fan and compatibility]
We say that a fan $\mathcal F$ is an $E$-fan if every cone $C\in\mathcal F$ is generated by rays determined by elements of $E$.

Let $\mathcal{F}$ be an $E$-fan. We say that $\mathcal{F}$ is compatible with the $\alpha E$-norm if $\rho(\, \cdot \, ; \alpha E)$ is linear on every $C \in \mathcal{F}$.
\end{definition}

The following result provides a local representation of the norm in terms of conical coordinates.

\begin{theorem}\label{theorem local fan form of alpha E norms}
Let $X$ be a $d$-dimensional Banach space, $E = (x_1 , \dots , x_n , -x_1 , \dots , -x_n)$ be a direction set and $\alpha = (\alpha_1 , \dots , \alpha_n)>0$ be an $E$-admissible multiindex. Then there exists a complete, symmetric, simplicial $E$-fan $\mathcal{F}$ compatible with $\rho(\, \cdot \, ; \alpha E)$. Moreover, if

\begin{equation}
C = \operatorname{cone}(x_{n_1} , \dots , x_{n_d}) \in \mathcal{F}
\end{equation}

\noindent is a maximal cone, then, for every 

\begin{equation}
x = \sum_{k=1}^d c_k x_{n_k}, \qquad c_k \geq 0
\end{equation}

\noindent one has

\begin{equation}
\rho(x  ; \alpha E) = \sum_{k= 1}^d c_k \alpha_{n_k}.
\end{equation}
\end{theorem}
\begin{proof}

We first construct a complete, symmetric, simplicial $E$-fan compatible with $\rho(\, \cdot \, ; \alpha E)$. Denote by $F_i$ the facets of $\operatorname{conv}(\alpha E)$. Since $\operatorname{conv}(\alpha E)$ is centrally symmetric, its facets come in antipodal pairs $F_i,-F_i$. We choose a triangulation of one representative of each antipodal pair and assign to the opposite facet the antipodal triangulation. Refining if necessary, we may assume that these triangulations agree on common faces.

\noindent Thus, for each pair $F_i,-F_i$, there exist $(d-1)$-triangulations $\mathcal{T}_i=\{\Delta^i_j\}_j$ of $F_i$ and $-\mathcal{T}_i=\{-\Delta^i_j\}_j$ of $-F_i$ which agree on common faces. Set $\mathcal{T} = \bigcup_i \, (\mathcal{T}_i \cup - \mathcal{T}_i)$.

Let $\Delta = \operatorname{conv}\{\alpha^{-1}_{n_1} x_{n_1} , \dots , \alpha^{-1}_{n_d} x_{n_d} \} \in \mathcal{T}$ be a maximal simplex. We claim that $x_{n_1} , \dots , x_{n_d}$ are linearly independent. Indeed, $\Delta$ is a simplex, then the vectors 

\begin{equation}
\alpha^{-1}_{n_k} x_{n_k} - \alpha^{-1}_{n_1} x_{n_1}, \qquad k = 2 , \dots , d
\end{equation}

\noindent are linearly independent. Since $\{\alpha^{-1}_{n_1} x_{n_1} , \dots , \alpha^{-1}_{n_d} x_{n_d} \}$ is contained in some facet, there exists a supporting functional $\ell$ such that $\ell(\alpha^{-1}_{n_k} x_{n_k})= 1$. Thus $\ell(\alpha^{-1}_{n_k} x_{n_k} - \alpha^{-1}_{n_1} x_{n_1}) = 0$. Hence $\alpha^{-1}_{n_1} x_{n_1} \notin \operatorname{span}(\alpha^{-1}_{n_k} x_{n_k} - \alpha^{-1}_{n_1} x_{n_1})$. That is, $\alpha^{-1}_{n_1}x_{n_1}$ is linearly independent of the vectors $\alpha^{-1}_{n_k} x_{n_k} - \alpha^{-1}_{n_1} x_{n_1}$, $k= 2 , \dots , d$. Therefore, 

\begin{equation}\label{technical linearly independence of vertices}
\sum_{k=2}^{d} c_k(\alpha^{-1}_{n_k}x_{n_k} - \alpha^{-1}_{n_1}x_{n_1}) + c_1 \alpha^{-1}_{n_1} x_{n_1} = 0
\end{equation}

\noindent implies $c_k = 0$ for each $k = 1 , \dots , d $. 

For each maximal simplex $\Delta = \operatorname{conv}\{\alpha^{-1}_{n_1} x_{n_1} , \dots , \alpha^{-1}_{n_d} x_{n_d} \}  \in \mathcal{T}$, define

\begin{equation}\label{technical E-fan equation}
C_\Delta=\operatorname{cone}(\Delta) = \operatorname{cone}(x_{n_1},\dots,x_{n_d}).
\end{equation}

\noindent Let $\mathcal{F}$ be the family formed by the cones $C_\Delta$ together with all their faces. Since the triangulations in $\mathcal{T}$ agree on common faces, the family $\mathcal{F}$ is a fan. Moreover, since $\mathcal{T}$ is symmetric and covers the facets of $\operatorname{conv}(\alpha E)$, the fan $\mathcal{F}$ is complete and symmetric. By (\ref{technical linearly independence of vertices}) and (\ref{technical E-fan equation}), every maximal cone of $\mathcal{F}$ is simplicial and generated by rays determined by elements of $E$. Hence $\mathcal{F}$ is a complete, symmetric, simplicial $E$-fan. The compatibility of $\mathcal{F}$ with $\rho(\, \cdot \, ; \alpha E)$ follows from the fact that every cone of $\mathcal{F}$ is contained in the cone generated by a facet of $\operatorname{conv}(\alpha E)$ on which $\rho(\,\cdot\,;\alpha E)$ is linear.

We now prove the piecewise linear representation. The case $x=0$ is trivial. Let $x \in X \setminus \{ 0\}$. Then there exists at least one maximal $C =\operatorname{cone}(x_{n_1} , \dots , x_{n_d}) \in \mathcal{F}$ with $x \in C$. Hence there exist unique $ c_1 , \dots , c_d \geq 0$ such that

\begin{equation}
x = \sum_{k= 1}^d c_k x_{n_k}.
\end{equation}

\noindent Then, by the compatibility of $\mathcal{F}$ with $\rho(\, \cdot \, ; \alpha E)$ and Lemma \ref{lemma properties of E-admissible multiindex}

\begin{equation}
\rho(x ; \alpha E) = \sum_{k= 1}^d c_k \rho(x_{n_k} ; \alpha E) =\sum_{k= 1}^d c_k \alpha_{n_k}.
\end{equation}
\end{proof}

The notion of an $\alpha E$-compatible fan captures the local piecewise linear regions of the norm. However, this structure alone is not enough to reconstruct the entire norm. For this reason, we introduce the following

\begin{definition}[$E$-admissible pair $(\mathcal{F}, \alpha)$]
Let $E$ be a direction set and $\mathcal{F}$ be a complete, symmetric, simplicial $E$-fan. Let $\alpha: E \to (0 , \infty)$ be a symmetric weight, that is, $\alpha(-x) =\alpha(x)$ for each $x \in E$. We say that the pair $(\mathcal F,\alpha)$ is $E$-admissible if, for every maximal cone

\begin{equation}
C=\operatorname{cone}(x_{n_1},\dots,x_{n_d})\in\mathcal F,
\end{equation}

\noindent the unique functional $\ell_C^\alpha\in X^*$ satisfying

\begin{equation}
\ell_C^\alpha (x_{n_k})=\alpha(x_{n_k}),\qquad k=1,\dots,d,
\end{equation}

\noindent also satisfies

\begin{equation}
\ell_C^\alpha(x_i)\leq \alpha(x_i)
\end{equation}

\noindent for every $x_i\in E$. 
\end{definition}

\begin{remark}\label{remark:weight multiindex identification}
If $E = (x_1 , \dots , x_n , -x_1, \dots , -x_n)$ then every symmetric weight $\alpha: E \to (0 , \infty)$ can be naturally identified with the multiindex 

\begin{equation}
\alpha' = (\alpha(x_1) , \dots , \alpha(x_n)) \in \mathbb{R}^n_{>0}.
\end{equation}

\noindent In what follows, we identify symmetric weights on $E$ with multiindices in $\mathbb R^n_{>0}$.
\end{remark}

\begin{example}
Let $X=\mathbb{R}^3$ and 

\begin{equation}
E=\{(\varepsilon_1,\varepsilon_2,\varepsilon_3)\mid \varepsilon_i\in\{-1,1\},\ i=1,2,3\}
\end{equation}

\noindent be the set of vertices of the cube $[-1,1]^3$. We distinguish the three positive facets

\begin{equation}
F_x^+=\{(1,\varepsilon_2,\varepsilon_3)\mid \varepsilon_2,\varepsilon_3\in\{-1,1\}\},
\qquad
F_y^+=\{(\varepsilon_1,1,\varepsilon_3)\mid \varepsilon_1,\varepsilon_3\in\{-1,1\}\},
\end{equation}

\noindent and

\begin{equation}
F_z^+=\{(\varepsilon_1,\varepsilon_2,1)\mid\varepsilon_1,\varepsilon_2\in\{-1,1\}\}.
\end{equation}

\noindent Each of these facets admits exactly two triangulations, according to the choice of one of its two diagonals. We denote by

\begin{equation}
d_r^\sigma, \qquad r=x,y,z, \qquad \sigma=1,2,
\end{equation}

\noindent the $\sigma$th diagonal of the positive facet $F_r^+$.

For each choice of diagonals

\begin{equation}
D=(d_x^{\sigma_1},d_y^{\sigma_2},d_z^{\sigma_3}), \qquad \sigma_1,\sigma_2,\sigma_3\in\{1,2\},
\end{equation}

\noindent we obtain a triangulation $\mathcal{T}_D$ of the positive facets. The triangulations of the opposite facets are obtained by antipodal symmetry. Thus $\mathcal{T}_D$ induces a symmetric triangulation of the boundary of the cube.

By taking the cones generated by the triangles of $\mathcal{T}_D$, together with all their faces, we obtain a complete, symmetric, simplicial $E$-fan, which we denote by $\mathcal{F}_D$. Since there are two possible choices of diagonal for each of the three distinguished facets, this construction gives

\begin{equation}
2^3=8
\end{equation}

\noindent different complete, symmetric, simplicial $E$-fans.
\end{example}

\begin{lemma}\label{lemma fan admissible induces norm}
Let $X$ be a $d$-dimensional Banach space, $E$ be a direction set and $\mathcal F$ be a complete, symmetric, simplicial $E$-fan. 

Let $\alpha:E\to(0,\infty)$ be a symmetric weight such that $(\mathcal F,\alpha)$ is $E$-admissible. For each maximal cone

\begin{equation}
C=\operatorname{cone}(x_{n_1},\dots,x_{n_d})\in\mathcal F,
\end{equation}

\noindent define $p_\alpha$ on $C$ by

\begin{equation}
p_\alpha(x)=\sum_{k=1}^d c_k\alpha(x_{n_k}),
\end{equation}

\noindent whenever

\begin{equation}
x=\sum_{k=1}^d c_k x_{n_k},\qquad c_k\geq 0.
\end{equation}

\noindent Then $p_\alpha$ is well defined and linear on every cone of $\mathcal F$. Moreover, the multiindex induced by $\alpha$ is $E$-admissible and

\begin{equation}
p_\alpha(x)=\rho(x;\alpha E),\qquad x\in X.
\end{equation}
\end{lemma}

\begin{proof}
Since $\mathcal{F}$ is complete, every $x\in X$ belongs to some maximal cone of $\mathcal{F}$. By Lemma \ref{lema full dimension of generators of maximal cones on fan}, every maximal cone is generated by $d$ linearly independent rays. Since $\mathcal{F}$ is simplicial, the conical coordinates of $x$ inside each maximal cone are unique.

We first prove that $p_\alpha$ is well defined. Suppose that $x\in C_1\cap C_2$, where $C_1$ and $C_2$ are maximal cones of $\mathcal{F}$. Since $\mathcal{F}$ is a fan, $C_1\cap C_2$ is a common face of both cones. Hence the representation of $x$ uses only the common rays of this face. Thus the corresponding coefficients are the same in both cones, and the values of $\alpha$ on the common generators are the same. Therefore the value of $p_\alpha(x)$ does not depend on the maximal cone chosen.

\noindent By construction, $p_\alpha$ is linear on every maximal cone, and hence on every cone of $\mathcal{F}$. Since $\mathcal{F}$ is symmetric and $\alpha$ is a symmetric weight, $p_\alpha$ is symmetric.

Now we prove that $p_\alpha$ is a polyhedral norm. For each maximal cone $C = \operatorname{cone}(x_{n_1},\dots,x_{n_d})\in\mathcal{F}$, let $\ell_C^\alpha\in X^*$ be the unique functional satisfying

\begin{equation}\label{technical equality 1}
\ell_C^\alpha(x_{n_k})=\alpha(x_{n_k}),\qquad k=1,\dots,d.
\end{equation}

\noindent Since $(\mathcal{F}, \alpha)$ is $E$-admissible, for every maximal cone $C$ and every $x_i\in E$ we have

\begin{equation}\label{technical equality 2}
\ell_C^\alpha(x_i)\leq \alpha(x_i).
\end{equation}

\noindent Let $x\in X$ and choose a maximal cone $D=\operatorname{cone}(x_{j_1},\dots,x_{j_d})$ containing $x$. Then there exist unique

\begin{equation}
x=\sum_{k=1}^d c_k x_{j_k},\qquad c_k \geq 0.
\end{equation}

\noindent Then, by (\ref{technical equality 1}) and (\ref{technical equality 2}), for every maximal cone $C$,

\begin{equation}
\begin{split}
\ell_C^\alpha(x) & = \sum_{k=1}^d c_k\ell_C^\alpha(x_{j_k}) \leq  \sum_{k=1}^d c_k\alpha(x_{j_k}) \\
                 & =  p_\alpha(x) = \ell_D^\alpha(x).
\end{split}
\end{equation}

\noindent Thus

\begin{equation}\label{technical dual form of p alpha}
p_\alpha(x)=\max_{C\in\mathcal{F}_{\max}}\ell_C^\alpha(x).
\end{equation}

\noindent where $\mathcal{F}_{\max}$ is the collection of maximal cones of $\mathcal{F}$. Since $\mathcal{F}$ is symmetric and $\alpha$ is symmetric, the family $\{\ell_C^\alpha: C\in\mathcal F_{\max}\}$ is symmetric. Hence $p_\alpha$ is the maximum of a finite symmetric family of linear functionals. By \eqref{polyhedral norm definition}, $p_\alpha$ is a polyhedral norm.

We now prove that $p_\alpha( \, \cdot \,) =\rho(\, \cdot \,;\alpha E)$. Let

\begin{equation}
P_\alpha=\operatorname{conv}\{\alpha(x)^{-1}x \mid x\in E\}, \qquad B_{p_\alpha} = \{x \in X \mid p_\alpha(x) \leq 1 \}.
\end{equation}

\noindent Since $(\mathcal{F}, \alpha )$ is $E$-admissible, by (\ref{technical equality 2}) for each $x \in E$ and each maximal cone $C \in \mathcal{F}_{\max}$ we have

\begin{equation}
\ell_C^\alpha(\alpha(x)^{-1}x)\leq 1.
\end{equation}

\noindent Therefore for each $x \in P_\alpha$ and $C \in \mathcal{F}_{\max}$ we have $\ell_C^\alpha(x) \leq 1$. Then using (\ref{technical dual form of p alpha}),

\begin{equation}
p_\alpha(x)=\max_{C\in\mathcal F_{\max}}\ell_C^\alpha(x) \leq 1
\end{equation}

\noindent for each $x \in P_\alpha$. Hence

\begin{equation}
P_\alpha\subseteq B_{p_\alpha}.
\end{equation}

\noindent Conversely, let $x\in B_{p_\alpha}$. Choose a maximal cone $D=\operatorname{cone}(x_{n_1},\dots,x_{n_d})$ such that $x\in D$. Then $x= \sum_{i=1}^d c_k x_{n_k}$ for some $c_k \geq 0$ and

\begin{equation}\label{technical p alfa leq 1}
\sum_{k=1}^d c_k\alpha(x_{n_k})= p_\alpha(x) \leq 1.
\end{equation}

\noindent Note that

\begin{equation}
x= \sum_{k=1}^d c_k\alpha(x_{n_k}) \left(\alpha(x_{n_k})^{-1}x_{n_k}\right).
\end{equation}

\noindent Thus by (\ref{technical p alfa leq 1}) and $0\in P_\alpha$, we get $x \in P_\alpha$. Hence

\begin{equation}
B_{p_\alpha}\subseteq P_\alpha.
\end{equation}

\noindent Then $p_\alpha$ is the Minkowski functional of $\operatorname{conv}(\alpha E)$, that is,

\begin{equation}
p_\alpha(x)=\rho(x;\alpha E),\qquad x\in X.
\end{equation}

Finally the $E$-admissibility of $\alpha$ viewed as a multiindex follows by Lemma \ref{lemma properties of E-admissible multiindex} and $\rho(x ; \alpha E) = p_\alpha(x) = \alpha(x)$ for each $x \in E$.
\end{proof}

\begin{lemma}\label{lemma support functional from linear cone}
Let $X$ be a $d$-dimensional Banach space, $p \in \mathcal{N}(X)$, $C= \operatorname{cone}(x_1 , \dots ,x_d)$ be a full dimensional cone such that $p$ is linear on $C$. Then there exists a unique supporting functional $\ell_C$ of $B_p$ such that 

\begin{equation}
\ell_C(x) = p(x) , \qquad x \in C
\end{equation} 

\noindent and

\begin{equation}
\ell_C(x) \leq p(x) , \qquad x \in X.
\end{equation}
\end{lemma}
\begin{proof}
Since $C$ is full dimensional and $p$ is linear on $C$, there exists a unique linear functional $\ell_C $ such that $\ell_C = p$ on $C$. We claim that $\ell_C(x)\leq 1$ for every $x\in B_p$. If this were not the case, there would exist $y\in B_p$ such that $\ell_C(y)>1$. Choose $u\in \operatorname{int}(C)$ with $p(u)=1$. Then $u \in B_p$ and $\ell_C(u)=1$. For $t>0$ sufficiently small, 

\begin{equation}
z_t=(1-t)u+ty \in C.
\end{equation}

\noindent On the other hand, since $B_p$ is convex and $u,y\in B_p$, we have $z_t\in B_p$, and hence $p(z_t)\leq 1$. However, since $z_t\in C$,

\begin{equation}
p(z_t)=\ell_C(z_t)=(1-t)\ell_C(u)+t\ell_C(y)>1,
\end{equation}

\noindent which is impossible. Thus $\ell_C(x)\leq 1$ for every $x\in B_p$. That is, $\ell_C$ is a support functional of $B_p$. 

Let $x \in X \setminus \{ 0 \}$. Then $p(x)^{-1} x \in B_p$ and

\begin{equation}
\ell_C(x) \leq p(x), \qquad x \in X. 
\end{equation}

\noindent In particular, by construction

\begin{equation}
\ell_C(x) = p(x), \qquad x \in C.
\end{equation}
\end{proof}

\begin{corollary}\label{corollary fan compatibility admissibility equivalence}
Let $X$ be a $d$-dimensional Banach space, $E$ be a direction set and $\mathcal{F}$ be a complete, symmetric, simplicial $E$-fan. Then the following statements are equivalent:

\begin{itemize}
\item[1)] $(\mathcal{F},\alpha)$ is $E$-admissible.
\item[2)] The $\alpha E$-norm $\rho(\, \cdot \,;\alpha E)$ is compatible with $\mathcal{F}$ and $\alpha$ is $E$-admissible.
\end{itemize}
\end{corollary}
\begin{proof}
By Lemma \ref{lemma fan admissible induces norm} the norm $\rho(\, \cdot \,;\alpha E)$ is compatible with $\mathcal{F}$ and $\alpha$ is $E$-admissible.

Conversely. By the $E$-admissibility of $\alpha$ and Lemma \ref{lemma properties of E-admissible multiindex} we have $\alpha(x_i) = \rho(x_i ; \alpha E)$ for each $x_i \in E$. Since $\rho(\, \cdot \, ; \alpha E)$ is compatible with $\mathcal{F}$, then it is linear on every maximal cone $C \in \mathcal{F}_{\max}$. Then by Lemma \ref{lemma support functional from linear cone} for every $C=\operatorname{cone}(x_{n_1} , \dots , x_{n_d}) \in \mathcal{F}_{\max}$ there exists a unique supporting functional $\ell_C^\alpha$ of $\operatorname{conv}(\alpha E)$ such that 

\begin{equation}
\ell_C^\alpha(x_{n_k}) = \rho(x_{n_k}; \alpha E) = \alpha(x_{n_k}), \qquad k= 1, \dots , d.
\end{equation}

\noindent and 

\begin{equation}
\ell_C^\alpha(x_i) \leq \rho(x_i; \alpha E) = \alpha(x_i), \qquad x_i \in E.
\end{equation}

\noindent Hence $(\mathcal{F}, \alpha)$ is $E$-admissible.
\end{proof}

\begin{lemma}\label{lemma rigidity of two dimensional structure of fans}
Let $X$ be a two-dimensional Banach space and $E$ be a direction set of cardinality $2n$. Then there exists a canonical complete, symmetric, simplicial $E$-fan $\mathcal{F}_E$ such that for every $E$-admissible multiindex $\alpha$, the pair $(\mathcal{F}_E, \alpha)$ is $E$-admissible. Consequently, every $\alpha E$-norm is compatible with $\mathcal{F}_E$.
\end{lemma}
\begin{proof}
First we construct the fan. Since $X$ is two-dimensional, after linearly identifying $X$ with $\mathbb{R}^2$, we may assume that the elements of $E=\{x_k \mid k = 1 , \dots, 2n \}$ are ordered cyclically according to their argument. That is,

\begin{equation}
\arg(x_i)<\arg(x_j)
\end{equation}

\noindent whenever $i<j$. Set $x_{2n + 1} = x_1$ and define

\begin{equation}
C_k=\operatorname{cone}(x_k,x_{k+1}), \qquad  k=1,\dots,2n.
\end{equation}

\noindent Let $\mathcal{F}_E$ be the collection of the cones $C_k$, $k=1,\dots,2n$, together with all their faces. Then $\mathcal{F}_E$ is a complete, symmetric, simplicial $E$-fan. Note that the maximal cones of $\mathcal{F}_E$ are the $C_k$ with $k= 1, \dots, 2n$.

Now we prove the admissibility. Let $\alpha$ be an $E$-admissible multiindex and $C_k=\operatorname{cone}(x_k,x_{k+1}) \in \mathcal{F}_E$ for some $k=1,\dots,2n$. For every $i= 1 , \dots, 2n +1$ we define 

\begin{equation}
v_i = \alpha(x_i)^{-1}x_i.
\end{equation}

\noindent Then there exists a unique functional $\ell_k$ such that 

\begin{equation}\label{technical functional on vertex}
\ell_k(v_k) = \ell_k(v_{k+1}) = 1.
\end{equation}

\noindent We claim that $\ell_k(v_i) \leq 1$ for each $i = 1, \dots, 2n$. Indeed, suppose that there exists some $v_l$ such that $\ell_k(v_l) > 1$. Then there are two possible cases: $v_l \in \operatorname{cone}(v_k , -v_{k+1})$ and $v_l \in \operatorname{cone}(-v_k , v_{k+1})$. We consider the first case; the second one is analogous. Using the $E$-admissibility of $\alpha$ and Lemma \ref{lemma properties of E-admissible multiindex} we have that 

\begin{equation}\label{technical norm one}
\rho(v_i ; \alpha E ) = 1 , \qquad i = 1 , \dots , 2n.
\end{equation} 

\noindent Since $v_l \in \operatorname{cone}(v_k , -v_{k+1})$ and $v_l \notin \operatorname{Dir}(\{v_k , -v_{k+1} \})$, there exist $a,b > 0$ such that

\begin{equation}\label{technical conic combination}
v_l=a v_k - b v_{k+1}.
\end{equation}

\noindent Using $\ell_k$ and (\ref{technical functional on vertex}) we have

\begin{equation}
1<\ell_k(v_l)=a - b.
\end{equation}

\noindent Hence 

\begin{equation}\label{technical lower bound for frac}
\frac{a}{1 + b} > 1
\end{equation}

\noindent By (\ref{technical conic combination}) we have 

\begin{equation}
\frac{a}{1 + b}v_k = \frac{b}{1 + b} v_{k+1} + \frac{1}{1 + b} v_l = : v^*. 
\end{equation}

\noindent Note that $\frac{b}{1 + b} , \frac{1}{1 + b} > 0 $ and $\frac{b}{1 + b} + \frac{1}{1 + b} = 1$. Therefore $v^*$ is a convex combination of $v_l$ and $v_{k+1}$. By (\ref{technical lower bound for frac}), there exists $\varepsilon>0$ such that

\begin{equation}\label{technical equal 1 + varepsilon}
(1 + \varepsilon)v_k = v^*.
\end{equation}

\noindent By the convexity of the norm and (\ref{technical norm one}) we have $\rho(v^* ; \alpha E) \leq 1$. Then, using (\ref{technical equal 1 + varepsilon}), we have $\rho(v_k ; \alpha E) \leq (1 + \varepsilon)^{-1 } < 1$, which is impossible. Thus no such $v_l$ exists, and therefore

\begin{equation}
\ell_k(v_i) \leq 1 , \qquad i= 1 , \dots , 2n.
\end{equation}

\noindent That is, for each maximal cone $C_k\in\mathcal{F}_E$, the functional $\ell_k$ satisfies

\begin{equation}
\ell_k(x_k)=\alpha(x_k), \qquad \ell_k(x_{k+1})=\alpha(x_{k+1}),
\end{equation}

\noindent and

\begin{equation}
\ell_k(x_i)\leq \alpha(x_i), \qquad x_i\in E.
\end{equation}

\noindent Hence $(\mathcal{F}_E,\alpha)$ is $E$-admissible for each $E$-admissible $\alpha$. 

Finally, the compatibility of $\rho(\, \cdot \, ; \alpha E)$ with $\mathcal{F}_E$ follows from Corollary \ref{corollary fan compatibility admissibility equivalence}.
\end{proof}

\begin{corollary}\label{corollary characterization (F,alpha) E-admissible}
Let $X$ be a $d$-dimensional Banach space and $E$ be a direction set.

\begin{itemize}
\item[1)] For every $\alpha E$-norm induced by an $E$-admissible multiindex, there exists a complete, symmetric, simplicial $E$-fan $\mathcal F$ compatible with $\rho(\, \cdot \,;\alpha E)$ such that $(\mathcal{F}, \alpha)$ is $E$-admissible.
\item[2)] Conversely, if $\mathcal{F}$ is a complete, symmetric, simplicial $E$-fan and $(\mathcal F,\alpha)$ is $E$-admissible, then the multiindex induced by $\alpha$ is $E$-admissible and the function $p_\alpha$ defined on each maximal cone by

\begin{equation}
x=\sum_{k=1}^d c_k x_{n_k} \quad\longmapsto\quad p_\alpha(x)=\sum_{k=1}^d c_k\alpha(x_{n_k})
\end{equation}

\noindent satisfies

\begin{equation}
p_\alpha(x)=\rho(x;\alpha E),\qquad x\in X.
\end{equation}
\end{itemize}
\end{corollary}
\begin{proof}

By Theorem \ref{theorem local fan form of alpha E norms} there exists a complete, symmetric, simplicial $E$-fan compatible with $\rho(\, \cdot \, ; \alpha E)$. The $E$-admissibility of $(\mathcal{F},\alpha)$ follows from Corollary \ref{corollary fan compatibility admissibility equivalence}.

Conversely, suppose that $(\mathcal F,\alpha)$ is $E$-admissible. Then Lemma \ref{lemma fan admissible induces norm} implies that the multiindex induced by $\alpha$ is $E$-admissible and that the function $p_\alpha$ defined on the maximal cones of $\mathcal F$ satisfies

\begin{equation}
p_\alpha(x)=\rho(x;\alpha E),\qquad x\in X.
\end{equation}
\end{proof}

\begin{corollary}\label{corollary fan linear implies E norm}
Let $X$ be a $d$-dimensional Banach space, $E=(x_1,\dots,x_n,-x_1,\dots,-x_n)$ be a direction set and $\mathcal{F}$ be a complete, symmetric, simplicial $E$-fan. If $p\in\mathcal{N}(X)$ is linear on every cone $C\in\mathcal{F}$, then there exists an $E$-admissible multiindex $\alpha$ such that

\begin{equation}
p(x)=\rho(x;\alpha E), \qquad x\in X.
\end{equation}

\noindent In particular,

\begin{equation}
\operatorname{Ext}(B_p)\subset \operatorname{Dir}(E).
\end{equation}
\end{corollary}
\begin{proof}
We define $\alpha_i=p(x_i)$ for every $i=1,\dots,n$ and $\alpha=(\alpha_i)$.

First we prove that $(\mathcal{F},\alpha)$ is $E$-admissible. By Lemma \ref{lemma support functional from linear cone} for every maximal cone $\operatorname{cone}(x_{n_1}, \dots , x_{n_d})\in\mathcal{F}$ there exists a unique functional $\ell_C$ such that 

\begin{equation}\label{technical functional equality}
\ell_C(x_{n_k}) = p(x_{n_k})= \alpha(x_{n_k}) , \qquad k=1 , \dots , d.
\end{equation} 

\noindent and

\begin{equation}
\ell_C(x_i) \leq p(x_i)=\alpha(x_i) , \qquad x_i \in E.
\end{equation}

\noindent Hence $(\mathcal{F} , \alpha)$ is $E$-admissible.

Now we prove the equality. By (\ref{technical functional equality}), for every maximal cone $C=\operatorname{cone}(x_{n_1}, \dots , x_{n_d})\in\mathcal{F}$ and $x = \sum_{k=1}^d c_k x_{n_k}$ we have 

\begin{equation}
p(x) = \sum_{k=1}^d c_k \alpha_{n_k}, \qquad x \in C.
\end{equation}

\noindent Thus $p$ coincides with the piecewise linear function $p_\alpha$ associated with $(\mathcal F,\alpha)$. Then by Lemma \ref{lemma fan admissible induces norm} 

\begin{equation}
p(x)=\rho(x;\alpha E), \qquad x\in X.
\end{equation}

Since $B_p=\operatorname{conv}(\alpha E)$, every extreme point of $B_p$ lies on a ray determined by an element of $E$. Then $\operatorname{Ext}(B_p)\subset \operatorname{Dir}(E)$.
\end{proof}

\begin{remark}
The structure of an $\alpha E$-norm can be understood through three layers: the combinatorial layer given by the set of directions $E$, the metric layer given by the weights $\alpha$, and the linearity layer encoded by the associated fan. 

While $E$ determines the available directions, it does not determine the regions of linearity. These are instead captured by the fan, which arises from the interaction between $E$ and $\alpha$.
\end{remark}

\section{Finite Direction, Coordinate, and Fan Models}

\begin{definition}[Finite direction and fan models]
Let $E$ be a direction set of cardinality $2n$. We define the finite direction model as

\begin{equation}
\mathcal{N}(E)=\{p\in \mathcal{N}(X) \mid \operatorname{Ext}(B_p)\subset \operatorname{Dir}(E)\}.
\end{equation}

\noindent This is the geometric model consisting of all polyhedral norms whose extreme points have directions prescribed by $E$. We also define the associated coordinate direction model by

\begin{equation}
\mathcal{R}(E)=\{\alpha\in \mathbb{R}^{n}_{>0} \mid \alpha \text{ is } E\text{-admissible}\}.
\end{equation}

\noindent Thus $\mathcal{R}(E)$ parametrizes the elements of $\mathcal{N}(E)$ by means of $E$-admissible multiindices.

Let $\mathcal{F}$ be a complete, symmetric, simplicial $E$-fan. We define the finite fan model as

\begin{equation}
\mathcal{N}(\mathcal{F})=\{p\in\mathcal{N}(X) \mid p \text{ is linear on every cone } C\in\mathcal{F}\}.
\end{equation}

\noindent This is the geometric submodel of $\mathcal{N}(E)$ obtained by fixing the conical regions of linearity. Similarly, we define the associated coordinate fan model by

\begin{equation}
\mathcal{R}(\mathcal{F})=\{\alpha:E\to(0,\infty) \mid \alpha(-x)=\alpha(x)\text{ and }(\mathcal{F},\alpha)\text{ is }E\text{-admissible}\}.
\end{equation}

\noindent Hence $\mathcal{R}(\mathcal{F})$ parametrizes the norms in $\mathcal{N}(\mathcal{F})$ through weights compatible with the fixed fan $\mathcal{F}$.
\end{definition}

\begin{lemma}\label{lemma parametrization NE NF}
Let $X$ be a finite-dimensional Banach space, $E$ be a direction set and $\mathcal F$ be a complete, symmetric, simplicial $E$-fan. Then

\begin{equation}
\mathcal{R}(\mathcal F)\subset \mathcal{R}(E) \qquad \text{and} \qquad \mathcal{N}(\mathcal F)\subset \mathcal{N}(E).
\end{equation}

\noindent Moreover, the map 

\begin{equation}
\phi_E:\mathcal{R}(E)\to \mathcal{N}(E), \qquad \phi_E(\alpha)=\rho(\,\cdot\,;\alpha E),
\end{equation}

\noindent is bijective. Its restriction

\begin{equation}
\phi_{\mathcal F}:\mathcal{R}(\mathcal F)\to \mathcal{N}(\mathcal F), \qquad \phi_{\mathcal F}(\alpha)=\rho(\,\cdot\,;\alpha E),
\end{equation}

\noindent is also bijective.
\end{lemma}

\begin{proof}
First we prove the inclusions. By Corollary \ref{corollary fan compatibility admissibility equivalence} if $(\mathcal{F}, \alpha)$ is $E$-admissible, then $\alpha$ is $E$-admissible. That is, 

\begin{equation}
\mathcal{R}(\mathcal{F}) \subset \mathcal{R}(E).
\end{equation}

\noindent On the other hand, If $p \in \mathcal{N}(\mathcal{F})$ then by Corollary \ref{corollary fan linear implies E norm}, $\operatorname{Ext}(B_p)\subset \operatorname{Dir}(E)$. That is,

\begin{equation}
\mathcal{N}(\mathcal{F}) \subset \mathcal{N}(E).
\end{equation}

Now we prove that $\phi_E$ is bijective. Let $p\in\mathcal{N}(E)$. Since the extreme points of $B_p$ lie in $\operatorname{Dir}(E)$, Theorem \ref{theorem fixed E representation} gives an $E$-admissible multiindex $\alpha$ such that

\begin{equation}
p(x)=\rho(x;\alpha E),\qquad x\in X.
\end{equation}

\noindent Hence $\phi_E$ is onto. The injectivity follows from Remark \ref{remark non redundancy in representation of alpha E norms}.

Now we prove the fan statement. By definition, $\alpha\in\mathcal{R}(\mathcal{F})$ implies $(\mathcal{F},\alpha)$ is $E$-admissible. Thus by Corollary \ref{corollary fan compatibility admissibility equivalence}, $\rho(\, \cdot \,;\alpha E)$ is linear on every cone of $\mathcal{F}$, that is,

\begin{equation}
\phi_{\mathcal{F}}(\alpha)=\rho(\, \cdot \,;\alpha E)\in\mathcal{N}(\mathcal{F}).
\end{equation}

\noindent Now we suppose that $p \in \mathcal{N}(\mathcal{F})$. By Corollary \ref{corollary fan linear implies E norm}, there exists an $E$-admissible multiindex $\alpha$ such that 

\begin{equation}
p(x) = \rho(x ; \alpha E), \qquad x \in X.
\end{equation}

\noindent Since $p$ is linear on every cone of $\mathcal{F}$, the norm $\rho(\, \cdot \, ; \alpha E)$ is compatible with $\mathcal{F}$. Hence, by Corollary \ref{corollary fan compatibility admissibility equivalence}, $(\mathcal{F}, \alpha)$ is $E$-admissible. That is, $\alpha \in \mathcal{R}(\mathcal{F})$ and

\begin{equation}
\phi_{\mathcal{F}}(\alpha)= \rho(\, \cdot \, ; \alpha E) = p(\, \cdot \, ).
\end{equation}

\noindent Therefore $\phi_{\mathcal{F}}$ is onto. Since $\phi_E$ is injective, its restriction $\phi_{\mathcal{F}}$ is also injective. Therefore $\phi_{\mathcal{F}}$ is a bijection.
\end{proof}

\begin{lemma}\label{lemma parameter sets are cones}
Let $X$ be a finite-dimensional Banach space, $E$ be a direction set of cardinality $2n$, and $\mathcal{F}$ be a complete, symmetric, simplicial $E$-fan. Then $\mathcal{R}(E)$ and $\mathcal{R}(\mathcal{F})$ are cones in $\mathbb{R}^n_{\geq 0}$ without the origin. More precisely, if $\alpha,\beta\in\mathcal{R}(E)$ and $c>0$, then

\begin{equation}
\alpha+\beta\in\mathcal{R}(E) \qquad\text{and}\qquad c \alpha\in\mathcal{R}(E).
\end{equation}

\noindent Similarly, if $\alpha,\beta\in\mathcal{R}(\mathcal{F})$ and $c>0$, then

\begin{equation}
\alpha+\beta\in\mathcal{R}(\mathcal{F})\qquad\text{and}\qquad c\alpha\in\mathcal{R}(\mathcal{F}).
\end{equation}
\end{lemma}
\begin{proof}
We assume $E = (x_1 , \dots, x_n , - x_1 , \dots , - x_n)$. 

First we prove that $\mathcal{R}(E)$ is a cone. Let $\alpha=(\alpha_1 , \dots , \alpha_n)$ and $\beta = (\beta_1 , \dots , \beta_n)$ be two $E$-admissible multi-indices and $c >0$. By Lemma \ref{lemma properties of E-admissible multiindex} for each $i  = 1 , \dots , n $ we have $\rho(x_i ; \alpha E) = \alpha_i$ and $\rho(x_i ; \beta E) = \beta_i$. Thus by the definition of the sum of functions $\| \cdot \| = \rho(\, \cdot \, ; \alpha E) + \rho(\, \cdot \, ; \beta E) $ satisfies 

\begin{equation}
\begin{split}
\| x_i \|  & = \rho(x_i ; \alpha E) +\rho(x_i ; \beta E) \\
           & = \alpha_i + \beta_i.
\end{split}
\end{equation}

\noindent Hence according to Lemma \ref{lemma properties of E-admissible multiindex}, $\alpha + \beta$ is an $E$-admissible multiindex. 

On the other hand, the $E$-admissibility of $\alpha$ implies $c \,\rho(x_i ; \alpha E) = c \alpha_i$ for every $i = 1 , \dots , n$. Then by Lemma \ref{lemma properties of E-admissible multiindex}, $c\alpha$ is an $E$-admissible multiindex.

Now we prove that $\mathcal{R}(\mathcal{F})$ is a cone. Let $\alpha , \beta \in \mathcal{R}(\mathcal{F})$ and $c > 0$. by the $E$-admissibility of $(\mathcal{F}, \alpha)$ and $(\mathcal{F}, \beta)$, for each maximal cone $C = \operatorname{cone}(x_{n_1} , \dots , x_{n_d}) \in \mathcal{F}$, there exists unique functionals $\ell_C^\alpha$ and $\ell_C^\beta$ such that 

\begin{equation}\label{technical supports for admissibility 1}
\ell_C^{\alpha}(x_{n_k}) = \alpha(x_{n_k}), \qquad  \ell_C^{\beta}(x_{n_k}) = \beta(x_{n_k}), \qquad k= 1 , \dots , d
\end{equation}

\noindent and

\begin{equation}\label{technical supports for admissibility 2}
\ell_C^{\alpha}(x_i) \leq \alpha(x_i),  \qquad \ell_C^{\beta}(x_i) \leq \beta(x_i), \qquad x_i \in E.
\end{equation}

\noindent Hence, for every maximal cone $C = \operatorname{cone}(x_{n_1} , \dots , x_{n_d}) \in \mathcal{F}$ we define $\ell_C^{\alpha + \beta} = \ell_C^\alpha + \ell_C^\beta$. Then

\begin{equation}
\ell_C^{\alpha + \beta}(x_{n_k}) = \alpha(x_{n_k}) + \beta(x_{n_k}), \qquad k= 1 , \dots , d
\end{equation}

\noindent and

\begin{equation}
\ell_C^{\alpha + \beta}(x_i) \leq \alpha(x_i) + \beta(x_i), \qquad x_i \in E.
\end{equation}

\noindent That is, $(\mathcal{F}, \alpha + \beta)$ is $E$-admissible and $\alpha + \beta \in \mathcal{R}(\mathcal{F})$. In the case of $c \alpha$, the proof is similar. We define for every maximal cone $\ell_C^{c \alpha} = c \, \ell_C^\alpha$. This family clearly satisfies (\ref{technical supports for admissibility 1}) and (\ref{technical supports for admissibility 2}). Hence $c \alpha \in \mathcal{R}(\mathcal{F})$. 
\end{proof}

\begin{theorem}\label{theorem cone structure of alpha E norms}
Let $X$ be a finite-dimensional Banach space, let $E$ be a direction set of cardinality $2n$, and let $\mathcal{F}$ be a complete, symmetric, simplicial $E$-fan. Then the following statements hold.

\begin{itemize}
\item[1)] For every $\alpha\in\mathcal{R}(E)$ and every $c>0$,

\begin{equation}
c \, \rho(\, \cdot \,;\alpha E)=\rho(\, \cdot \,;c\alpha E).
\end{equation}

\noindent In particular, $\mathcal{N}(E)$ is closed under positive scalar multiplication.

\item[2)] The set $\mathcal{N}(\mathcal{F})$ is a cone in $\mathcal{N}(X)$, with respect to pointwise addition and positive scalar multiplication of norms. Moreover, the map

\begin{equation}
\phi_{\mathcal{F}}:\mathcal{R}(\mathcal{F}) \to \mathcal{N}(\mathcal{F}), \qquad \phi_{\mathcal{F}}(\alpha)=\rho(\, \cdot \,; \alpha E),
\end{equation}

\noindent is an isomorphism of cones. Equivalently, for every $\alpha, \beta \in \mathcal{R}(\mathcal{F})$ and every $c>0$,

\begin{equation}
\rho(\, \cdot \, ;(\alpha + \beta ) E) = \rho(\, \cdot \,; \alpha E) + \rho(\, \cdot \,; \beta E),
\end{equation}

\noindent and

\begin{equation}
\rho(\, \cdot \,; c \alpha E) = c \, \rho(\, \cdot \,; \alpha E).
\end{equation}

\item[3)] If $\dim X=2$, then $\mathcal{N}(E)$ is a cone in $\mathcal{N}(X)$, and the map

\begin{equation}
\phi_E:\mathcal{R}(E)\to\mathcal{N}(E), \qquad \phi_E(\alpha)=\rho(\, \cdot \,;\alpha E),
\end{equation}

\noindent is an isomorphism of cones. Equivalently, for every $\alpha,\beta \in \mathcal{R}(E)$ and every $c>0$,

\begin{equation}
\rho(\, \cdot \,; (\alpha+\beta) E) = \rho(\, \cdot \,; \alpha E) + \rho(\, \cdot \,; \beta E),
\end{equation}

\noindent and

\begin{equation}
\rho(\, \cdot \,; c \alpha E) = c\, \rho(\, \cdot \,;\alpha E).
\end{equation}
\end{itemize}
\end{theorem}
\begin{proof}
$1)$. Recall that $c\alpha E = \{\pm c^{-1} \alpha_1^{-1} x_1 , \dots , \pm c^{-1} \alpha_n^{-1} x_n \}$, therefore

\begin{equation}
\operatorname{conv}(c \alpha E) = c^{-1} \operatorname{conv}(\alpha E).
\end{equation}

\noindent Then, by the definition of the Minkowski functional and setting $s = r c^{-1}$, for each $x \in X$ we obtain

\begin{equation}\label{technical homogeneity of phi_F(alpha)}
\begin{split}
\rho(x;c\alpha E) &= \inf\{r>0 \mid x\in r\,\operatorname{conv}(c\alpha E)\}  \\
                  &= \inf\{r>0 \mid x\in r c^{-1}\operatorname{conv}(\alpha E)\} \\
                  &= c \,\inf\{s>0 \mid x\in s \operatorname{conv}(\alpha E)\} \\
                  &= c \,\rho(x;\alpha E).
\end{split}
\end{equation}

$2)$ Let $\alpha , \beta \in  \mathcal{R}(\mathcal{F})$ and $c> 0$. By Lemma \ref{lemma parameter sets are cones}, $c \alpha \in \mathcal{R}(\mathcal{F})$. Thus using (\ref{technical homogeneity of phi_F(alpha)}) we have

\begin{equation}\label{technical homogeneity of phi_F}
\phi_\mathcal{F}(c \alpha) = \rho(\, \cdot \, ; c \alpha E) = c \rho(\, \cdot \, ; \alpha E) = c \phi_\mathcal{F}(\alpha).
\end{equation}

\noindent By the completeness of $\mathcal{F}$ and Lemma \ref{lemma fan admissible induces norm}, for each $x \in X$ there exist a maximal cone $C = \operatorname{cone}(x_{n_1} , \dots , x_{n_d}) \in \mathcal{F}$ and coordinates $c_1 , \dots , c_d \geq 0$ such that $x = \sum_{k=1}^d c_k x_{n_k} $,

\begin{equation}\label{technical PL representation 1}
\rho(x; \alpha E) = \sum_{k=1}^d c_k \alpha(x_{n_k}), \qquad \rho(x ; \beta E) = \sum_{k=1}^d c_k \beta(x_{n_k}).
\end{equation}

\noindent Since Lemma \ref{lemma parameter sets are cones} implies $\alpha + \beta \in \mathcal{R}(\mathcal{F})$, Lemma \ref{lemma fan admissible induces norm} gives

\begin{equation}\label{technical PL representation 2}
\begin{split}
\rho(x ; (\alpha+ \beta) E) & = \sum_{k=1}^d c_k (\alpha + \beta)(x_{n_k})\\
                            & = \sum_{k=1}^d c_k (\alpha(x_{n_k}) + \beta(x_{n_k})).
\end{split}
\end{equation}

\noindent By (\ref{technical PL representation 1}) and (\ref{technical PL representation 2}) we obtain

\begin{equation}\label{technical additivity of phi_F}
\phi_\mathcal{F}(\alpha + \beta)= \phi_\mathcal{F}(\alpha) + \phi_\mathcal{F}(\beta)
\end{equation}

\noindent By (\ref{technical homogeneity of phi_F}) and (\ref{technical additivity of phi_F}), $\phi_\mathcal{F}$ preserves addition and positive scalar multiplication. By Lemma \ref{lemma parametrization NE NF}, $\phi_\mathcal{F}$ is a bijection. Therefore $\phi_\mathcal{F}$ is an isomorphism of cones between $\mathcal{R}(\mathcal{F})$ and $\mathcal{N}(\mathcal{F})$.

$3)$. We observe that it is enough to prove that $\mathcal{N}(E)=\mathcal{N}(\mathcal{F}_E)$ for a suitable complete, symmetric, simplicial $E$-fan $\mathcal{F}_E$. 

Since $X$ is two-dimensional, by Lemma \ref{lemma rigidity of two dimensional structure of fans}, there exists a canonical complete, symmetric, simplicial $E$-fan $\mathcal{F}_E$, such that for every $\alpha \in \mathcal{R}(E)$, the pair $(\mathcal{F}_E, \alpha)$ is $E$-admissible. Then, by Corollary \ref{corollary fan compatibility admissibility equivalence}, $\rho(\, \cdot \, ; \alpha E)$ is compatible with $\mathcal{F}_E$ for each $\alpha \in \mathcal{R}(E)$.

Let $p\in\mathcal{N}(E)$. By Lemma \ref{lemma parametrization NE NF}, there exists $\alpha\in\mathcal{R}(E)$ such that

\begin{equation}
p=\rho(\,\cdot\,;\alpha E).
\end{equation}

\noindent Hence $p$ is compatible with $\mathcal{F}_E$, and therefore $p\in\mathcal{N}(\mathcal{F}_E)$. Thus

\begin{equation}
\mathcal{N}(E)\subset \mathcal{N}(\mathcal{F}_E).
\end{equation}

\noindent Since the opposite inclusion follows from Lemma \ref{lemma parametrization NE NF}, we conclude that

\begin{equation}
\mathcal{N}(E)=\mathcal{N}(\mathcal{F}_E).
\end{equation}
\end{proof}

\begin{definition}
Let $\sim$ be the collinearity relation, that is, $x \sim y$ if there exists $c > 0$ such that $y = c x$. We define 

\begin{equation}
\mathcal{R}'(E) = \mathcal{R}(E) / \sim
\end{equation}

\noindent and 

\begin{equation}
\mathcal{N}'(E) = \mathcal{N}(E)/ \sim
\end{equation}
\end{definition}

In the following theorem we use the Hilbert $d_H$ and $d$ metrics defined on preliminaries section. 

\begin{theorem}\label{theorem isometric models}
Let $X$ be a finite-dimensional Banach space and $E$ be a direction set. Then the map

\begin{equation}
\phi'_E: (\mathcal{R}'(E), d_H) \to (\mathcal{N}'(E),d) , \qquad \phi'_E(\overline{\alpha}) = \overline{\rho(\, \cdot \, ; \alpha E)}
\end{equation}

\noindent is an isometry. Moreover, if $\mathcal{F}$ is a complete, symmetric, simplicial $E$-fan, then the map

\begin{equation}
\phi'_\mathcal{F}:(\mathcal{R}'(\mathcal{F}),d_H) \to (\mathcal{N}'(\mathcal{F}),d) , \qquad \phi'_\mathcal{F}(\overline{\alpha}) = \overline{\rho(\, \cdot \, ; \alpha E)}
\end{equation}

\noindent is an isometry.
\end{theorem}
\begin{proof}
Let $\phi'_E:\mathcal{R}'(E)\to \mathcal{N}'(E)$ and $\phi'_\mathcal{F} : \mathcal{R}'(\mathcal{F})\to \mathcal{N}'(\mathcal{F})$ be defined by

\begin{equation} 
\phi'_E(\overline{\alpha})=\overline{\rho(\, \cdot \,;\alpha E)}, \qquad \phi'_\mathcal{F}(\overline{\alpha})=\overline{\rho(\, \cdot \,;\alpha E)}.
\end{equation}

\noindent By Lemma \ref{lemma parametrization NE NF}, the maps $\phi_E$ and $\phi_\mathcal{F}$ are bijective. Moreover, by Theorem \ref{theorem cone structure of alpha E norms}, they preserve positive scalar multiplication. Therefore, $\phi'_E$ and $\phi'_\mathcal{F}$ are well defined and bijective. That is, the following diagrams are commutative

\begin{equation}
\begin{tikzcd}
\mathcal{R}(E) \arrow[r, "\phi_E"] \arrow[d, "\sim"'] & \mathcal{N}(E) \arrow[d, "\sim"] \\
\mathcal{R}'(E) \arrow[r, "\phi_E'"']                                 & \mathcal{N}'(E)
\end{tikzcd}
\qquad \qquad
\begin{tikzcd}
\mathcal{R}(\mathcal{F}) \arrow[r, "\phi_{\mathcal{F}}"] \arrow[d, "\sim"'] & \mathcal{N}(\mathcal{F}) \arrow[d, "\sim"] \\
\mathcal{R}'(\mathcal{F}) \arrow[r, "\phi_{\mathcal{F}}'"'] & \mathcal{N}'(\mathcal{F})
\end{tikzcd}
\end{equation}

\noindent By Lemma \ref{lemma parametrization NE NF} we have that $\phi_\mathcal{F}$ is the restriction of $\phi_E$ to $\mathcal{R}(\mathcal{F})$. Thus $\phi'_\mathcal{F}$ is the restriction of $\phi'_E$ to $\mathcal{R}'(\mathcal{F})$. In diagram form

\begin{equation}
\begin{tikzcd}
\mathcal{R}'(\mathcal{F}) \arrow[r, "\phi'_{\mathcal{F}}"] \arrow[d, hook] & \mathcal{N}'(\mathcal{F}) \arrow[d, hook] \\
\mathcal{R}'(E) \arrow[r, "\phi'_E"] & \mathcal{N}'(E)
\end{tikzcd}
\end{equation}

\noindent Hence, it is enough to prove the isometry of $\phi'_E$ to ensure that of $\phi'_\mathcal{F}$.

We assume $E= (x_1 , \dots, x_n , -x_1, \dots, - x_n)$. Let $\alpha , \beta \in \mathcal{R}(E)$ with $\alpha = (\alpha_1 , \dots, \alpha_n)$ and $\beta = (\beta_1 , \dots, \beta_n)$. Set

\begin{equation}
m=\min_i \frac{\alpha_i}{\beta_i},\qquad M=\max_i \frac{\alpha_i}{\beta_i}.
\end{equation}

First we prove $d \leq d_H$. Since

\begin{equation}
m\beta_i \leq \alpha_i \leq M\beta_i,
\end{equation}

\noindent then

\begin{equation}
[0, m \alpha^{-1}_i x_i] \subset[0, \beta^{-1}_i x _i] \subset [0, M \alpha^{-1}_i x_i].
\end{equation}

\noindent Hence 

\begin{equation}
m\,\operatorname{conv}(\alpha E)\subset \operatorname{conv}(\beta E)\subset M\,\operatorname{conv}(\alpha E).
\end{equation}

\noindent Thus passing to Minkowski functionals, we obtain

\begin{equation}
\frac{1}{M}\rho(\, \cdot \,;\alpha E) \leq \rho(\, \cdot \, ;\beta E) \leq \frac{1}{m}\rho(\, \cdot \, ;\alpha E).
\end{equation}

\noindent Hence

\begin{equation}
d\left(\phi'_E(\overline{\alpha}),\phi'_E(\overline{\beta})\right) \leq \log\left(\frac{M}{m}\right)=d_H(\overline{\alpha},\overline{\beta}).
\end{equation}

Now we prove $d_H \leq d$. Let $u \geq l>0$ be the optimal constants such that

\begin{equation}
l\,\rho(\, \cdot \, ;\alpha E)\leq \rho(\, \cdot \, ;\beta E)\leq u\,\rho(\, \cdot \, ;\alpha E).
\end{equation}

\noindent By Lemma \ref{lemma properties of E-admissible multiindex} we have

\begin{equation}
\rho(x_i;\alpha E)=\alpha_i,\qquad \rho(x_i;\beta E)=\beta_i,
\end{equation}

\noindent thus

\begin{equation}
l\alpha_i\leq \beta_i\leq u\alpha_i.
\end{equation}

\noindent Therefore

\begin{equation}
l\leq \min_i \frac{\beta_i}{\alpha_i},\qquad \max_i \frac{\beta_i}{\alpha_i}\leq u.
\end{equation}

\noindent Then

\begin{equation}
l\leq \frac{1}{M},\qquad \frac{1}{m}\leq u.
\end{equation} 

\noindent Hence

\begin{equation}
d_H(\overline{\alpha},\overline{\beta}) = \log\left(\frac{M}{m}\right) \leq  \log\left(\frac{u}{l}\right) =d \left(\phi'_E(\overline{\alpha}),\phi'_E(\overline{\beta})\right).
\end{equation}

\noindent We conclude that

\begin{equation}
d \left(\phi'_E(\overline{\alpha}),\phi'_E(\overline{\beta})\right) = d_H(\overline{\alpha},\overline{\beta}).
\end{equation}
\end{proof}

\section{Completeness}

In this section we establish the completeness of the finite-dimensional models and show that they are compatible with the global metric structure.

\begin{lemma}\label{lemma:pointwise convergence of alpha implies admissibility}
Let $X$ be a finite-dimensional Banach space and $E$ be a direction set of cardinality $2k$. Let $(\alpha_n)$ be a sequence of $E$-admissible multiindices such that

\begin{equation}
\alpha_i^n \to \alpha_i > 0 , \qquad  i= 1,\dots,k.
\end{equation}

\noindent Then the multiindex $\alpha=(\alpha_1,\dots,\alpha_k)$ is $E$-admissible and satisfies

\begin{equation}
\rho(x;\alpha_n E) \to \rho(x;\alpha E), \quad x\in X.
\end{equation}
\end{lemma}
\begin{proof}
Let $0 < \varepsilon < 1$. Since $\alpha_i^n\to \alpha_i$ for each $i= 1 , \dots , k$, there exists $N\in\mathbb{N}$ such that for all $n\geq N$

\begin{equation}
(1-\varepsilon)\alpha_i \leq \alpha_i^n \leq (1+\varepsilon)\alpha_i , \qquad  i = 1,\dots,k.
\end{equation}

\noindent Hence

\begin{equation}
[0 , (1 + \varepsilon)^{-1}\alpha^{-1}_i x_i] \subset [0 , (\alpha^n_i)^{-1} x_i] \subset [0 , (1 - \varepsilon)^{-1} \alpha^{-1}_i x_i]
\end{equation}

\noindent for each $n \geq N$ and $i = 1 , \dots, k$, Then

\begin{equation}
\frac{1}{1+\varepsilon}\operatorname{conv}(\alpha E) \subset \operatorname{conv}(\alpha_nE) \subset \frac{1}{1-\varepsilon}\operatorname{conv}(\alpha E).
\end{equation}

\noindent Passing to Minkowski functionals, we obtain

\begin{equation}\label{technical:alpha inq}
(1-\varepsilon)\rho(\,\cdot\,;\alpha E) \leq \rho(\,\cdot\,;\alpha_nE) \leq (1+\varepsilon)\rho(\,\cdot\,;\alpha E),
\end{equation}

\noindent for each $n \geq N$. Evaluating at $x_i$, and using Lemma \ref{lemma properties of E-admissible multiindex} and the $E$-admissibility of $\alpha_n$ for each $n \geq N$, we get

\begin{equation}
(1-\varepsilon)\rho(x_i;\alpha E) \leq \rho(x_i;\alpha_nE) =\alpha_i^n\leq(1+\varepsilon)\rho(x_i;\alpha E).
\end{equation}

\noindent Letting $n \to \infty$, we obtain

\begin{equation}
(1-\varepsilon)\rho(x_i;\alpha E)\leq\alpha_i\leq(1+\varepsilon)\rho(x_i;\alpha E).
\end{equation}

\noindent Since $\varepsilon\in(0,1)$ is arbitrary, it follows that

\begin{equation}
\rho(x_i;\alpha E)=\alpha_i,  \qquad i =1,\dots,k.
\end{equation}

\noindent Hence, by Lemma \ref{lemma properties of E-admissible multiindex}, the multiindex $\alpha$ is $E$-admissible. Moreover, for each fixed $x\in X$, from (\ref{technical:alpha inq}) we have

\begin{equation}
(1-\varepsilon)\rho(x;\alpha E)\leq\rho(x;\alpha_nE)\leq(1+\varepsilon)\rho(x;\alpha E)
\end{equation}

\noindent for all $n \geq N$. Then

\begin{equation}
\rho(x;\alpha_nE)\to \rho(x;\alpha E).
\end{equation}
\end{proof}

The following result shows that each finite-dimensional model is complete with respect to the metric $d$.

\begin{theorem}\label{theorema: R'(E) is d_H-complete}
Let $X$ be a finite-dimensional Banach space and $E$ be a direction set of cardinality $2k$. Then $(\mathcal{R}'(E), d_H)$ and $(\mathcal{N}'(E),d)$ are complete metric spaces.
\end{theorem}
\begin{proof}
By Theorem \ref{theorem isometric models} the families $(\mathcal{R}'(E), d_H)$ and $(\mathcal{N}'(E) , d)$ are isometric, thus it is enough to prove the completeness of $(\mathcal{R}'(E), d_H)$.

Since $(\mathcal{R}'(E), d_H) \subset (\mathbb{P}(\mathbb{R}^k_{>0}) , d_H)$ and $(\mathbb{P}(\mathbb{R}^k_{>0}) , d_H)$ is complete by Proposition \ref{proposition completeness of projective positive cone}, it is enough to prove that $\mathcal{R}'(E)$ is closed in $\mathbb{P}(\mathbb{R}^k_{>0})$.

Let $(\overline{\alpha}_n)$ be a sequence in $(\mathcal{R}'(E),d_H)$ converging to $\overline{\alpha}$. By Proposition \ref{proposition normalized representatives converge} we can choose normalized representatives $\alpha_n=(\alpha_1^n,\dots,\alpha_k^n)\in \overline{\alpha}_n$ and $\alpha=(\alpha_1,\dots,\alpha_k)$ such that 

\begin{equation}
\lim_{n \to \infty} \alpha_i^n = \alpha_i , \qquad i = 1 , \dots, k.
\end{equation}

\noindent By Lemma \ref{lemma parameter sets are cones}, $\mathcal{R}(E)$ is a cone, then every representative $\alpha_n$ is $E$-admissible. Since $\alpha \in \mathbb{P}(\mathbb{R}^k_{>0})$, we have $\alpha_i > 0$ for each $i = 1 , \dots , k$. By Lemma \ref{lemma:pointwise convergence of alpha implies admissibility} the multiindex $\alpha$ is $E$-admissible. That is, $\overline{\alpha} \in \mathcal{R}'(E)$. Therefore $(\mathcal{R}'(E),d_H)$ is complete.
\end{proof}

\begin{lemma}\label{lemma coordinate convergence preserves fan admissibility}
Let $X$ be a $d$-dimensional Banach space, $E$ be a direction set and $\mathcal{F}$ be a complete, symmetric, simplicial $E$-fan. Let $(\alpha_n)$ be a sequence in $\mathcal{R}(\mathcal{F})$ such that

\begin{equation}
\alpha_n(x_i)\to \alpha(x_i)>0, \qquad x_i\in E.
\end{equation}

\noindent Then $\alpha\in\mathcal{R}(\mathcal{F})$.
\end{lemma}
\begin{proof}
Since each $\alpha_n$ is symmetric, the pointwise limit $\alpha$ is also symmetric. 

Let $C=\operatorname{cone}(x_{n_1},\dots,x_{n_d})$ be a maximal cone of $\mathcal{F}$. Since $\alpha_n\in\mathcal{R}(\mathcal{F})$, for each $n \in \mathbb{N}$ there exists a unique functional $\ell_C^{\alpha_n}$ such that

\begin{equation}
\ell_C^{\alpha_n}(x_{n_k})=\alpha_n(x_{n_k}),\qquad k=1,\dots,d,
\end{equation}

\noindent and

\begin{equation}\label{technical inequality of functionals}
\ell_C^{\alpha_n}(x_i)\leq \alpha_n(x_i), \qquad x_i\in E.
\end{equation}

\noindent Let $\ell_C^\alpha$ be the unique functional satisfying

\begin{equation}
\ell_C^\alpha(x_{n_k})=\alpha(x_{n_k}), \qquad k=1,\dots,d.
\end{equation}

\noindent Since $C$ is maximal and $\mathcal{F}$ is simplicial, by Lemma \ref{lema full dimension of generators of maximal cones on fan} $x_{n_1},\dots,x_{n_d}$ form a basis of $X$. For each $x_i\in E$ we can write $x_i=\sum_{k=1}^d a_{i,k}x_{n_k}$. Hence

\begin{equation}
\ell_C^{\alpha_n}(x_i) = \sum_{k=1}^d a_{i,k}\alpha_n(x_{n_k})\to\sum_{k=1}^d a_{i,k}\alpha(x_{n_k})=\ell_C^\alpha(x_i).
\end{equation}

\noindent Passing to the limit in (\ref{technical inequality of functionals}) we obtain

\begin{equation}
\ell_C^\alpha(x_i)\leq \alpha(x_i), \qquad x_i\in E.
\end{equation}

\noindent Since $C$ was arbitrary, $(\mathcal{F},\alpha)$ is $E$-admissible and $\alpha\in\mathcal{R}(\mathcal{F})$.
\end{proof}

\begin{corollary}
Let $X$ be a finite-dimensional Banach space, $E$ be a direction set of cardinality $2k$ and $\mathcal{F}$ be a complete, symmetric, simplicial $E$-fan. Then $(\mathcal{R}'(\mathcal{F}), d_H)$ and $(\mathcal{N}'(\mathcal{F}),d)$ are complete metric spaces.
\end{corollary}
\begin{proof}
By Theorem \ref{theorema: R'(E) is d_H-complete}, $(\mathcal{R}'(E), d_H)$ is complete. Since $(\mathcal{R}'(\mathcal{F}), d_H) \subset (\mathcal{R}'(E), d_H)$, then it is enough to prove the closedness of $(\mathcal{R}'(\mathcal{F}), d_H)$.

Let $\overline{\alpha}_n$ be a sequence in $(\mathcal{R}'(\mathcal{F}), d_H)$ converging to $\overline{\alpha} \in (\mathcal{R}'(E) , d_H)$. By Proposition \ref{proposition normalized representatives converge} we can choose normalized representatives $\alpha_n \in \overline{\alpha}_n$ and $\alpha \in \overline{\alpha}$ such that 

\begin{equation}
\lim_{n\to \infty} \alpha_i^n = \alpha_i , \qquad i =1, \dots, k.
\end{equation}

\noindent $\mathcal{R}(\mathcal{F})$ is a cone by Lemma \ref{lemma parameter sets are cones}, thus we may assume that every $\alpha_n \in \mathcal{R}(\mathcal{F})$. By Lemma \ref{lemma coordinate convergence preserves fan admissibility} we have that $\alpha \in \mathcal{R}(\mathcal{F})$. Hence $(\mathcal{R}'(\mathcal{F}), d_H)$ is complete.

Since $(\mathcal{R}'(\mathcal{F}), d_H)$ and $(\mathcal{N}'(\mathcal{F}),d)$ are isometric by Theorem \ref{theorem isometric models}, we have that $(\mathcal{N}'(\mathcal{F}),d)$ is complete.
\end{proof}

The next theorem shows that the finite polyhedral models are dense in the global space of equivalent norms.

\begin{theorem}
Let $X$ be a finite-dimensional Banach space. Then the union of the spaces $\mathcal{N}'(E)$, where $E$ ranges over all direction sets, is $d$-dense in $\mathcal{N}'(X)$.
\end{theorem}
\begin{proof}
Let $\overline{\| \cdot \|} \in \mathcal{N}'(X)$, $\| \cdot \| \in \overline{\| \cdot \|}$, $\varepsilon_0> 0$ and $0 < \varepsilon < 1$ such that

\begin{equation}
\log\left(\frac{1 + \varepsilon}{1 - \varepsilon}\right)< \varepsilon_0
\end{equation}

\noindent Since polyhedral norms are uniformly dense in the class of equivalent norms on finite-dimensional spaces, there exists a polyhedral norm $p$ such that 

\begin{equation}\label{uniform distance bound for polyhedral approximation}
\sup \{ | p( x) - \| x \| \mid x \in S_{\| \cdot \|} \} < \varepsilon.
\end{equation}

\noindent By Corollary \ref{corollary representation of polyhedral norms} there exists a direction set $E$ and an $E$-admissible multiindex $\alpha$ such that
 
\begin{equation}
p(x) = \rho(x ; \alpha E)
\end{equation} 

\noindent for each $x \in X$. Then by (\ref{uniform distance bound for polyhedral approximation}) we have

\begin{equation}
1 - \varepsilon \leq \rho(x ; \alpha E) \leq 1 + \varepsilon
\end{equation}

\noindent for each $x \in S_{\| \cdot \|}$. Hence by the homogeneity of the norm

\begin{equation}
(1 - \varepsilon) \| x \| \leq \rho(x ;\alpha E) \leq (1 + \varepsilon) \| x \|.
\end{equation}

\noindent Finally

\begin{equation}
d\bigl(\overline{\rho(\, \cdot \, ;\alpha E)},\overline{\| \cdot \|}\bigr) \leq \log\left(\frac{1 + \varepsilon}{1 - \varepsilon}\right) < \varepsilon_0.
\end{equation}
\end{proof}

\subsection{Final comments}
The preceding results show that the finite polyhedral models introduced in this paper play two complementary roles. On the one hand, for each fixed direction set $E$, the space $\mathcal{N}'(E)$ is a complete metric submodel of $\mathcal{N}'(X)$, and its metric structure is exactly described by the Hilbert projective metric on the admissible coordinate space $\mathcal{R}'(E)$. On the other hand, as $E$ ranges over all direction sets, these complete finite models form a dense family in $\mathcal{N}'(X)$.

Thus the global space of equivalent norms can be approached through finite-dimensional polyhedral data: directions, admissible radial weights, and compatible fans. The direction set records the possible locations of the extreme points of the unit ball, the admissible weights provide non-redundant coordinates, and the fan describes the local regions of linearity of the corresponding polyhedral norm. In this sense, the models $\mathcal{N}'(E)$ and $\mathcal{N}'(\mathcal{F})$ provide finite metric charts for the study of $\mathcal{N}'(X)$.

\end{document}